\documentclass[11pt,a4paper,leqno]{article}
\usepackage{a4wide}
\usepackage{latexsym}
\usepackage{amsmath}
\usepackage{amsthm}
\usepackage{amssymb,enumerate}
\theoremstyle{plain}
\newtheorem{theo}{Theorem}[section]
\newtheorem{lemm}[theo]{Lemma}

\newtheorem{prop}[theo]{Proposition}
\newtheorem{coro}[theo]{Corollary}
\theoremstyle{definition}
\newtheorem{defi}[theo]{Definition}
\newtheorem{rema}[theo]{Remark}
\newtheorem{exam}[theo]{Example}

\newenvironment{proof1}{\medskip\par\noindent{\bf Proof.}}{\hfill $\Box$
\medskip\par}

\def\C{\mathbb{C}}
\def\N{\mathbb{N}}
\def\R{\mathbb{R}}
\def\Z{\mathbb{Z}}

\def\a{\alpha}
\def\b{\beta}
\def\ro{\rho}
\def\ga{\gamma}

\def\o{\omega}
\def\bmu{\boldsymbol{\mu}}

\def\bm{\boldsymbol{m}}

\def\bM{\mathbb{M}}
\def\M{\mathbb{M}}

\def\parn{\par\noindent}

\begin{document}

\title{Summability in general Carleman ultraholomorphic classes}
\author{Alberto Lastra, St\'ephane Malek and Javier Sanz}
\date{\today}

\maketitle

{ \small
\begin{center}
{\bf Abstract}
\end{center}
A definition of summability is put forward in the framework of general Carleman ultraholomorphic classes in sectors, so generalizing $k-$summability theory as developed by J.-P. Ramis. Departing from a strongly regular sequence of positive numbers, we construct an associated analytic proximate order and corresponding kernels, which allow us to consider suitable Laplace and Borel-type transforms, both formal and analytic, whose behavior closely resembles that of the classical ones in the Gevrey case. An application to the study of the summability properties of the formal solutions to some moment-partial differential equations is included.

}

\bigskip

\section{Introduction}

\indent\indent
The aim of this paper is to put forward a concept of summability of formal (i.e. divergent in general) power series with controlled growth in their coefficients in the framework of general Carleman ultraholomorphic classes in sectors, so generalizing the by-now classical and powerful tool of $k$-summability of formal Gevrey power series, introduced by J.-P. Ramis~\cite{Ramis1,Ramis2}.

Given a sequence of positive real numbers  $\M=(M_n)_{n\in\N_0}$, the Carleman ultraholomorphic class $\tilde{\mathcal{A}}_{\M}(G)$ in a sectorial region $G$ of the Riemann surface of the logarithm consists of those holomorphic functions $f$ in $G$ whose derivatives of order $n\ge 0$ are bounded on every bounded proper subsector of $G$ by, essentially, the values $n!M_n$. Equivalently (see Subsection~\ref{subsectCarlemanclasses}), these functions admit a (non-uniform) $\M$-asymptotic expansion at 0 in $G$, given by a formal power series $\hat f=\sum_{n\ge 0}a_nz^n/n!$ whose coefficients are again suitably bounded in terms of $\M$ (we write $f\sim_{\M}\hat f$ and $(a_n)_{n\in\N_0}\in\Lambda_{\M}$; $f$ is said to be flat if $\tilde{\mathcal{B}}(f)$ is the null sequence). The map sending $f$ to $(a_n)_{n\in\N_0}$ is the asymptotic Borel map $\tilde{\mathcal{B}}$, and its injectivity in $\tilde{\mathcal{A}}_{\M}(G)$ means that every function $f$ in such a class is determined by its asymptotic expansion $\hat f$ (in other words, the class does not contain nontrivial flat functions). In this case, the class $\tilde{\mathcal{A}}_{\M}(G)$ is said to be quasianalytic, and it makes sense to call $f$ the sum of $\hat f$; this is the idea behind summability methods in this context.

We will only consider strongly regular sequences $\M$ as defined by V. Thilliez~\cite{thilliez}, which are subject to standard conditions guaranteeing good properties for the considered classes (see Subsection~\ref{subsectstrregseq}). The best known example is that of Gevrey classes, corresponding to $\bM_{\a}=(n!^{\a})_{n\in\N_{0}}$, $\a>0$, for which we use the notations $\tilde{\mathcal{A}}_{\a}(G)$,  $\Lambda_{\a}$, $f\sim_{\a}\hat{f}$ and so on, for simplicity.
Let us denote by $S_\ga$ the sector bisected by the direction $d=0$ and with opening $\pi\ga$.
It is well known that $\tilde{\mathcal{B}}:\tilde{\mathcal{A}}_{\a}(S_\ga)\to\Lambda_{\a}$ is injective if, and only if, $\ga>\a$ (Watson's Lemma, see for example~\cite[Prop.\ 11]{balserutx}).
This result is the departure point for the definition of $1/\a-$summability in a direction, introduced by J.-P. Ramis~\cite{Ramis1,Ramis2}. A paradigmatic example of flat function in $\tilde{\mathcal{A}}_{\a}(S_{\a})$ is $f_{\a}(z)=\exp(-z^{-1/\a})$, and it gives rise to kernels in terms of which one may define formal and analytic Laplace and Borel transforms permitting the reconstruction of the sum of a given Gevrey formal power series belonging to $\tilde{\mathcal{B}}(\tilde{\mathcal{A}}_{\a}(S_{\ga}))$ for some $\ga>\a$.
The technique of multisummability (in a sense, an iteration of a finite number of $1/\a-$summability procedures) has been proven to apply successfully to a plethora of situations concerning the study of formal power series solutions at a singular point for linear and nonlinear (systems of) meromorphic ordinary differential equations in the complex domain (see, to cite but a few, the works~\cite{Balser,balserutx,BalserBraaksmaRamisSibuya,Braaksma,MartinetRamis,RamisSibuya}), for partial differential equations (for example, \cite{Balsermulti,BalserMiyake,Hibino,Malek3,Ouchi}), as well as for singular perturbation problems (see~\cite{BalserMozo,CanalisMozoSchafke,lastramaleksanz2}, among others).

However, it is known that non-Gevrey formal power series solutions may appear for different kinds of equations. For example, V. Thilliez has proven some results on solutions within these general classes for algebraic equations in~\cite{Thilliez2}. Also, G. K. Immink in~\cite{Immink,Immink2} has obtained some results on summability for solutions of difference equations whose coefficients grow at an intermediate rate between  Gevrey classes, called of $1^+$ level, that is governed by a strongly regular sequence. Very recently, the second author~\cite{Malek4} has studied some singularly perturbed small step size difference-differential nonlinear equations whose formal solutions with respect to the perturbation parameter can be decomposed as sums of two formal series, one with Gevrey order 1, the other of $1^+$ level, a phenomenon already observed for difference equations~\cite{BraaksmaFaberImmink}.

All these results invite one to try to extend summability tools so that they are able to deal with formal power series whose coefficients' growth is controlled by a general strongly regular sequence, so including Gevrey, $1^+$ level and other interesting examples. Our approach will be inspired by the study of moment summability methods, equivalent in a sense to $1/\a-$summability, developed by W.~Balser in~\cite[Section\ 5.5]{balserutx} and which relies on the determination of a pair of kernel functions with suitable asymptotic and growth properties, in terms of which to define formal and analytic Laplace- and Borel-like transforms.
These summability methods have already found its application to the analysis of formal power series solutions of different classes of partial differential equations (for example, by the second author~\cite{Malek1,Malek2} and by S. Michalik~\cite{Michalik1,Michalik2}), and also for so-called moment-partial differential equations, introduced by W. Balser and Y. Yoshino~\cite{BalserYoshino} and subsequently studied by S. Michalik~\cite{Michalik3,Michalik4}.

It seems clear that an analogue of Watson's lemma should be obtained for a proper definition of summability, and flat functions in sectors of optimal opening (like $f_{\a}$, see above) should be determined. We recall that V. Thilliez~\cite{thilliez} introduced a growth index $\ga(\M)\in(0,\infty)$ for every strongly regular sequence $\M$ (which for $\M_{\a}$ equals $\a$), and proved the following facts: if $\ga<\ga(\M)$, then $\mathcal{A}_{\M}(S_{\ga})$ contains nontrivial flat functions; and, resting on Whitney-type extension results, he obtained a generalized Borel-Ritt-Gevrey theorem
(coming with right linear and continuous inverses for the Borel map). By means of these flat functions, the authors~\cite{lastramaleksanz} defined suitable kernels and moment sequences and reproved this last result by the classical truncated Laplace transform technique, in the same vein as Ramis' original proof. In the last section of~\cite{lastramaleksanz}, and resting on some partial Watson-like results obtained by the first and third authors~\cite{lastrasanz1}, some hints are given on how to define a Laplace transform which behaves properly with respect to $\M-$asymptotics, and a definition of $\M-$summability is suggested. However, the preceding results for general classes are not fully satisfactory, as the quantity $\ga(\M)$ is not known to be optimal for quasianalyticity and, moreover, flat functions are only obtained after restricting ourselves to sectors $S_\ga$ with $\ga<\ga(\M)$.

These drawbacks have been recently overcome in the paper~\cite{sanz13} by the third author.
There, for every strongly regular sequence $\M$ a new constant $\omega(\M)$ is introduced, measuring the rate of growth of the sequence $\M$, in terms of which quasianalyticity may be properly characterized due to a classical result of B. I. Korenbljum~\cite{korenbljum}. This constant is easily computed (see~(\ref{equaordequasM}) and (\ref{equaordeMdet})), and indeed it is the inverse of the order of growth of the classical function $M(t)$ associated with $\M$, namely
$M(t)=\sup_{n\in\N_0}\log(t^n/M_n)$, $t>0$. So, a definition of $\M-$summability in a direction $d$ may be easily put forward, see Definition~\ref{defisumable}.
Regarding the construction of flat functions, V. Thilliez \cite{Thilliez2} had characterized flatness in $\tilde{\mathcal{A}}_{\M}(S_{\ga})$ in terms of non-uniform estimates governed by the function $e^{-M(1/|z|)}$, much in the same way as the function $e^{-|z|^{-1/\a}}$ expresses flatness in the Gevrey case.
The theory of proximate orders, which refines the notion of constant exponential order, allows one to specify the rate of growth of an entire function in terms of the function $M(t)$, and results by V. Bernstein, M. M. Dzhrbashyan, M. A. Evgrafov, A. A. Gol'dberg, I. V. Ostrovskii (see \cite{Levin,GoldbergOstrowskii}) and, in our regards, mainly L. S. Maergoiz~\cite{Maergoiz}, have been the key for the construction of flat functions in $S_{\omega(\M)}$, whenever $\M$ induces a proximate order (which is the case in all the instances of strongly regular sequences appearing in the literature). Section~\ref{sectkernelsfromproxorder} is devoted to these results. We mention that proximate orders had already been used in the study of some questions regarding multisummability, see~\cite{BalserBraun}.

Now that we have assured its existence under fairly mild assumptions, we devote Section~\ref{sectMsummability} to the introduction of kernels of $\M-$summability (see Definition~\ref{defikernelMsumm}), and the associated formal and analytic transforms, in terms of which to reconstruct the sums of $\M-$summable formal power series in a direction, as stated in Theorem~\ref{teorsumableequivTsumable}. Once our tool has been designed, it is necessary to test its applicability to the study of formal solutions of different types of algebraic and differential equations in the complex domain. Our first attempt is contained in the last section of the paper.
The notion of formal moment-differential operator was firstly introduced by W. Balser and M. Yoshino in~\cite{BalserYoshino}. Generally speaking, given the sequence of moments $\mathfrak{m}_e=(m_e(p))_{p\in\N_0}$ of a kernel function $e$ of order $k>0$ (in other words and according to Remark \ref{remanotassumabM}(i), a kernel for $\M_{1/k}-$summability), one can define $\partial_{\mathfrak{m}_e,z}$ as an operator from $\C[[z]]$ into itself given by
\begin{equation*}
\partial_{\mathfrak{m}_e,z}\left(\sum_{p\ge0}\frac{f_{p}}{m_e(p)}z^{p}\right)=
\sum_{p\ge0}\frac{f_{p+1}}{m_e(p)}z^{p},
\end{equation*}
in much the same way as, for the usual derivative $\partial$, one has $\partial\left(\sum_{p\ge0}\frac{f_{p}}{p!}z^{p}\right)=\sum_{p\ge0}\frac{f_{p+1}}{p!}z^{p}$.
For two sequences of moments $\mathfrak{m}_1=(m_1(p))_{p\in\N_0}$ and $\mathfrak{m}_2=(m_2(p))_{p\in\N_0}$ of orders $k_1$ and $k_2$, respectively, they study the Gevrey order of the formal power series solutions of an inhomogeneous moment-partial differential equation with constant coefficients in two variables,
$$
p(\partial_{\mathfrak{m}_1,t},\partial_{\mathfrak{m}_2,z})\hat{u}(t,z)=\hat{f}(t,z),
$$
where $p(\lambda,\xi)$ is a given polynomial.
Subsequently, S. Michalik~\cite{Michalik3} considers the corresponding initial value problem
$$
P(\partial_{\mathfrak{m}_1,t},\partial_{\mathfrak{m}_2,z})u(t,z)=0,\qquad \partial^j_{\mathfrak{m}_1,t}u(0,z)=\varphi_j(z)\quad\textrm{ for }j=0,\ldots,n-1,
$$
where $P(\lambda,\xi)$ is a polynomial of degree $n$ with respect to $\lambda$ and the Cauchy data are analytic in a neighborhood of $0\in\C$.
A formal solution $\hat u$ is constructed, and a detailed study is made of the relationship between the summability properties of $\hat u$ and the analytic continuation properties and growth estimates for the Cauchy data. We will generalize his results for strongly regular moment sequences of a general kernel of summability. A significant part of the statements are given without proof, since the arguments do not greatly differ from those in~\cite{Michalik3}. On the other hand, complete details are provided when the differences between both situations are worth stressing.

\section{Preliminaries}
\subsection{Notation}\label{notation}
We set $\N:=\{1,2,\ldots\}$, $\N_{0}:=\N\cup\{0\}$.
$\mathcal{R}$ stands for the Riemann surface of the logarithm, and
$\C[[z]]$ is the space of formal power series in $z$ with complex coefficients.\par\noindent
For $\gamma>0$, we consider unbounded sectors
$$S_{\gamma}:=\{z\in\mathcal{R}:|\hbox{arg}(z)|<\frac{\gamma\,\pi}{2}\}$$
or, in general, bounded or unbounded sectors
$$S(d,\alpha,r):=\{z\in\mathcal{R}:|\hbox{arg}(z)-d|<\frac{\alpha\,\pi}{2},\ |z|<r\},\quad
S(d,\alpha):=\{z\in\mathcal{R}:|\hbox{arg}(z)-d|<\frac{\alpha\,\pi}{2}\}$$
with bisecting direction $d\in\R$, opening $\alpha\,\pi$ and (in the first case) radius $r\in(0,\infty)$.\par\noindent
A sectorial region $G(d,\a)$ with bisecting direction $d\in\R$ and opening $\alpha\,\pi$ will be a domain in $\mathcal{R}$ such that $G(d,\a)\subset S(d,\a)$, and
for every $\beta\in(0,\a)$ there exists $\rho=\rho(\beta)>0$ with $S(d,\beta,\rho)\subset G(d,\a)$. In particular, sectors are sectorial regions.\par\noindent
A sector $T$ is a bounded proper subsector of a sectorial region $G$ (denoted by $T\ll G$) whenever the radius of $T$ is finite and $\overline{T}\setminus\{0\}\subset G$.
Given two unbounded sectors $T$ and $S$, we say $T$ is an unbounded proper subsector of $S$, and we write $T\prec S$, if $\overline{T}\setminus\{0\}\subset S$.\par\noindent
$D(z_0,r)$ stands for the disk centered at $z_0$ with radius $r>0$.\par\noindent
For an open set $U\subset\mathcal{R}$, $\mathcal{O}(U)$ denotes the set of holomorphic functions defined in $U$.\par\noindent
$\Re(z)$ stands for the real part of a complex number $z$, and we write $\left\lfloor x\right\rfloor$ for the integer part of $x\in\R$, i.e. the greatest integer not exceeding $x$.

\subsection{Strongly regular sequences}\label{subsectstrregseq}

Most of the information in this subsection is taken from the works of A. A. Goldberg and I. V. Ostrovskii~\cite{GoldbergOstrowskii}, H. Komatsu~\cite{komatsu}, V. Thilliez~\cite{thilliez} and the third author~\cite{sanz13}, which we refer to for further details and proofs.
In what follows, $\bM=(M_p)_{p\in\N_0}$ will always stand for a sequence of
positive real numbers, and we will always assume that $M_0=1$.
\begin{defi}\label{defisucfuereg}
We say $\bM$ is \textit{strongly regular} if the following hold:\par
($\a_0$) $\bM$ is \textit{logarithmically convex}: $M_{p}^{2}\le M_{p-1}M_{p+1}$ for
every $p\in\N$.\par
($\mu$) $\bM$ is of \textit{moderate growth}: there exists $A>0$ such that
$$M_{p+\ell}\le A^{p+\ell}M_{p}M_{\ell},\qquad p,\ell\in\N_0.$$
\par($\gamma_1$) $\bM$ satisfies the \textit{strong non-quasianalyticity condition}: there exists $B>0$ such that
$$
\sum_{\ell\ge p}\frac{M_{\ell}}{(\ell+1)M_{\ell+1}}\le B\frac{M_{p}}{M_{p+1}},\qquad p\in\N_0.
$$
\end{defi}

%\begin{rema}
%If $\M$ is strongly regular, then $\M'=(n!M_n)_{n\in\N_0}$ verifies (M.1)+(M.2)+(M.3) of H. Komatsu.\parn
%If $\M'=(M'_n)_{n\in\N_0}$ verifies (M.2)+(M.3) of H. Komatsu and $\M=(M'_n/n!)_{n\in\N_0}$ is logarithmically convex, then $\M$ is strongly regular.
%\end{rema}

\begin{exam}\label{examstroregusequ}
\begin{itemize}
\item[(i)] The best known example of strongly regular sequence is $\M_{\a}:=(n!^{\a})_{n\in\N_{0}}$, called the \textit{Gevrey sequence of order $\a>0$}.
\item[(ii)] The sequences $\M_{\a,\b}:=\big(n!^{\a}\prod_{m=0}^n\log^{\b}(e+m)\big)_{n\in\N_0}$, where $\a>0$ and $\b\in\R$, are strongly regular.
\item[(iii)] For $q>1$, $\M=(q^{n^2})_{n\in\N_0}$ satisfies $(\alpha_0)$ and $(\gamma_{1})$, but not $(\mu)$.
    \end{itemize}
\end{exam}

For a sequence $\bM=(M_{p})_{p\in\N_0}$ verifying properties $(\alpha_0)$ and $(\gamma_{1})$ one has that the associated \textit{sequence of quotients}, $\bm=(m_{p}:=M_{p+1}/M_{p})_{p\in\N_0}$, is an increasing sequence to infinity, so that the map $h_{\bM}:[0,\infty)\to\R$, defined by
\begin{equation}\label{equadefihdeM}
h_{\bM}(t):=\inf_{p\in\N_{0}}M_{p}t^p,\quad t>0;\qquad h_{\bM}(0)=0,
\end{equation}
turns out to be a non-decreasing continuous map in $[0,\infty)$ onto $[0,1]$. In fact
$$
h_{\bM}(t)= \left \{ \begin{matrix}  M_{p}t^{p} & \mbox{if }t\in\left[\frac{1}{m_{p}}\right. ,\left. \frac{1}{m_{p-1}}\right),\ p=1,2,\ldots,\\
1 & \mbox{if } t\ge 1/m_{0}. \end{matrix}\right.
$$

\begin{defi}[\cite{pet},~\cite{chaucho}]
Two sequences $\bM=(M_{p})_{p\in\N_0}$ and $\bM'=(M'_{p})_{p\in\N_0}$ of positive real numbers are said to be \textit{equivalent}
%, and we write $\bM\sim\bM'$,
if there exist positive constants $L,H$ such that
$$L^pM_p\le M'_p\le H^pM_p,\qquad p\in\N_0.$$
\end{defi}
In this case, it is straightforward to check that
\begin{equation}\label{equahdeMequi}
h_{\bM}(Lt)\le h_{\bM'}(t)\le h_{\bM}(Ht),\qquad t\ge 0.
\end{equation}

One may also associate with such a sequence $\bM$ the function
\begin{equation}\label{equadefiMdet}
M(t):=\sup_{p\in\N_{0}}\log\big(\frac{t^p}{M_{p}}\big)=-\log\big(h_{\bM}(1/t)\big),\quad t>0;\qquad M(0)=0,
\end{equation}
which is a non-decreasing continuous map in $[0,\infty)$ with $\lim_{t\to\infty}M(t)=\infty$.\parn

We now recall the following definitions and facts.

\begin{defi}[\cite{GoldbergOstrowskii}, p.\ 43]\label{defiordefunc}
Let $\a(r)$ be a nonnegative and nondecreasing function in $(c,\infty)$ for some $c\ge 0$.
% (we write $\a\in\Lambda$).
The \textit{order} of $\a$ is defined as
$$
\rho=\rho[\a]:=\limsup_{r\to\infty}\frac{\log^+(\a(r))}{\log(r)}\in[0,\infty]
$$
(where $\log^+=\max(\log,0)$).
%$\a(r)$ is said to have finite order if $\rho<\infty$, and in this case the (magnitude of the) type is defined as
%$$
%\sigma=\sigma[\a]:=\limsup_{r\to\infty}\frac{\a(r)}{r^{\rho}}\in[0,\infty].
%$$
\end{defi}

\begin{theo}[\cite{sanz13}]
Let $\M$ be strongly regular, $\bm$ the sequence of its quotients and $M(r)$ its associated function. Then, the order of $M(r)$ is given by
\begin{equation}\label{equaordeMdet}
\rho[M]=\lim_{r\to\infty}\frac{\log( M(r))}{\log(r)}=\limsup_{n\to\infty}\frac{\log(n)}{\log(m_{n})}\in(0,\infty).
\end{equation}
\end{theo}

%\begin{rema}\label{remadefifuncdder}
%The function $d$, defined for $r>\max\{1,m_0\}$ by $d(r)=\log(M(r))/\log r$, will play an important role in Section~\ref{sectkernelsfromproxorder}.
%\end{rema}

Some additional properties of strongly regular sequences needed in the present work are the following ones.
\begin{lemm}[\cite{thilliez}]
Let $\bM=(M_{p})_{p\in\N_0}$ be a strongly regular sequence and $A>0$ the constant appearing in $(\mu)$. Then,
\begin{equation}\label{equameneequivMene}
M_{p}^{1/p}\le m_{p}\le A^{2}M_{p}^{1/p}\qquad \hbox{for every }p\in\N_0.
\end{equation}
%\begin{align}
%\label{propgrowth}
%&M_{p+\ell}\ge M_pM_{\ell},\qquad \hbox{for every }p,\ell\in\N_0,\\
%\label{equameneequivMene}
%&M_{p}^{1/p}\le m_{p}\le A^{2}M_{p}^{1/p},\qquad \hbox{for every }p\in\N_0.
%%\\
%%\label{e115}
%%&m_{p},\qquad \hbox{for every }p\in\N_0.
%\end{align}
Let $s$ be a real number with $s\ge1$. There exists $\rho(s)\ge1$ (only depending on $s$ and $\M$) such that
\begin{equation}\label{e120}
h_{\bM}(t)\le(h_{\bM}(\rho(s)t))^{s}\qquad\hbox{for }t\ge0.
\end{equation}
\end{lemm}

\begin{rema}
The condition of moderate growth $(\mu)$ plays a fundamental role in the proof of (\ref{e120}), which will in turn be crucial in many of our arguments.
\end{rema}

For the sake of completeness we include the following result.

\begin{prop}\label{propprodfuertregu}
Given two strongly regular sequences $\M=(M_n)_{n\in\N_0}$ and $\M'=(M'_n)_{n\in\N_0}$, the product sequence $\M\cdot\M':=(M_nM'_n)_{n\in\N_0}$ is also strongly regular.
\end{prop}

\begin{proof1}
Properties ($\a_0$) and ($\mu$) are easily checked for $\M\cdot\M'$. Regarding ($\gamma_1$),
we will use that, according to Lemma 1.3.4 in Thilliez~\cite{thilliez}, $\bM^s:=(M_n^s)_{n\in\N_0}$ is strongly regular for every $s>0$, hence $\M^2$ and ${\M'}^2$ are. Now, Cauchy-Schwarz inequality gives that, for every $p\ge 0$,
$$
\sum_{\ell\ge p}\frac{M_{\ell}M'_{\ell}}{(\ell+1)M_{\ell+1}M'_{\ell+1}}\le
\Big(\sum_{\ell\ge p}\frac{M^2_{\ell}}{(\ell+1)M^2_{\ell+1}}\Big)^{1/2}
\Big(\sum_{\ell\ge p}\frac{{M'}^2_{\ell}}{(\ell+1){M'}^2_{\ell+1}}\Big)^{1/2}\le
\sqrt{BB'}\frac{M_{p}}{M_{p+1}}\frac{M'_{p}}{M'_{p+1}},
$$
where the positive constants $B$ and $B'$ are the ones appearing in ($\gamma_1$) for $\M^2$ and ${\M'}^2$, respectively.
\end{proof1}

\subsection{Asymptotic expansions and ultraholomorphic classes}\label{subsectCarlemanclasses}

Given a sequence of positive real numbers $\M=(M_n)_{n\in\N_0}$, a constant $A>0$ and a sector $S$, we define
$$\mathcal{A}_{\M,A}(S)=\big\{f\in\mathcal{O}(S):\left\|f\right\|_{\M,A}:=\sup_{z\in S,n\in\N_{0}}\frac{|f^{(n)}(z)|}{A^{n}n!M_{n}}<\infty\big\}.$$
%\begin{itemize}
%\item $f\in\mathcal{A}_{\M,A}(S): |f^{(n)}(z)|\le\left\|f\right\|_{\M,A}A^{n}n!M_n$, $n\in\N_{0},z\in S$.
%\item
($\mathcal{A}_{\M,A}(S),\left\| \ \right\| _{\M,A}$) is a Banach space, and $\mathcal{A}_{\M}(S):=\cup_{A>0}\mathcal{A}_{\M,A}(S)$ is called a \textit{Carleman ultraholomorphic class} in the sector $S$.
\parn
One may accordingly define classes of sequences
$$\Lambda_{\M,A}=\Big\{\bmu=(\mu_{n})_{n\in\N_{0}}\in\C^{\N_{0}}: \left|\bmu\right|_{\M,A}:=\sup_{n\in\N_{0}}\displaystyle \frac{|\mu_{n}|}{A^{n}n!M_{n}}<\infty\Big\}.$$
$(\Lambda_{\M,A},\left| \  \right|_{\M,A})$ is again a Banach space, and we put $\Lambda_{\M}:=\cup_{A>0}\Lambda_{\M,A}$.\parn
%\end{itemize}
Since the derivatives of $f\in\mathcal{A}_{\bM,A}(S)$ are Lipschitzian, for every $n\in\N_{0}$ one may define
$$f^{(n)}(0):=\lim_{z\in S,z\to0 }f^{(n)}(z)\in\C,$$
and it is clear that the sequence
\begin{equation*}%\label{equaBoremap}
\tilde{\mathcal{B}}(f):=(f^{(n)}(0))_{n\in\N_{0}}\in\Lambda_{\M,A},\qquad f\in\mathcal{A}_{\bM,A}(S).
\end{equation*}
The map $\tilde{\mathcal{B}}:\mathcal{A}_{\M}(S)\longrightarrow \Lambda_{\M}$
so defined is the \textit{asymptotic Borel map}.

Next, we will recall the relationship between these classes and the concept of asymptotic expansion.
\begin{defi}
We say a holomorphic function $f$ in a sectorial region $G$ admits the formal power series $\hat{f}=\sum_{p=0}^{\infty}a_{p}z^{p}\in\C[[z]]$ as its \textit{$\bM-$asymptotic expansion} in $G$ (when the variable tends to 0) if for every $T\ll G$ there exist $C_T,A_T>0$ such that for every $n\in\N$, one has
\begin{equation*}\Big|f(z)-\sum_{p=0}^{n-1}a_pz^p \Big|\le C_TA_T^nM_{n}|z|^n,\qquad z\in T.%\label{desarasint}
\end{equation*}
We will write $f\sim_{\bM}\sum_{p=0}^{\infty}a_pz^p$ in $G$. $\tilde{\mathcal{A}}_{\M}(G)$ stands for the space of functions admitting $\bM-$asymptotic expansion in $G$.\par
\end{defi}

\begin{defi}
Given a sector $S$, we say $f\in\mathcal{O}(S)$ admits $\hat{f}$ as its \textit{uniform $\bM-$asymptotic expansion in $S$ of type $A>0$} if there exists $C>0$ such that for every $n\in\N$, one has
\begin{equation*}\Big|f(z)-\sum_{p=0}^{n-1}a_pz^p \Big|\le CA^nM_{n}|z|^n,\qquad z\in S.%\label{desarasintunifo}
\end{equation*}
\end{defi}

As a consequence of Taylor's formula and Cauchy's integral formula for the derivatives, we have the following result (see \cite{balserutx,galindosanz}).
\begin{prop}\label{propcotaderidesaasin}
Let $S$ be a sector and $G$ a sectorial region.
\begin{itemize}
\item[(i)] If $f\in\mathcal{A}_{\M,A}(S)$, then $f$ admits $\hat{f}=\sum_{p\in\N_0}\frac{1}{p!}f^{(p)}(0)z^p$ as its uniform $\M-$asymptotic expansion in $S$ of type $A$.
\item[(ii)] $f\in\tilde{\mathcal{A}}_{\M}(G)$ if, and only if, for every $T\ll G$ there exists $A_T>0$ such that $f|_T\in \mathcal{A}_{\M,A_T}(T)$. Hence, the map $\tilde{\mathcal{B}}:\tilde{\mathcal{A}}_{\M}(G)\longrightarrow \Lambda_{\M}$ is also well defined.
\end{itemize}
\end{prop}

%\begin{rema}\label{remaCarlclassasympexpan}
%Indeed, for a sector $S$ one can prove that whenever $T\ll S$, there exists a constant $c=c(T,S)>0$ such that the restriction to $T$, $f_T$, of functions $f$ defined on $S$ and admitting uniform $\bM-$asymptotic expansion in $S$ of type $A>0$, belongs to $\mathcal{A}_{\bM,cA}(T)$, and moreover, if one has (\ref{desarasintunifo}) then $\Vert f_T\Vert_{\bM,cA}\le C$.
%\end{rema}

\begin{rema}\label{remaasymspacequisequ}
Consider a pair of equivalent sequences $\bM$ and $\bM'$. It is obvious that $\Lambda_{\M}=\Lambda_{\M'}$; given a sector $S$, the spaces $\mathcal{A}_{\bM}(S)$ and $\mathcal{A}_{\bM'}(S)$ coincide, and
also $\tilde{\mathcal{A}}_{\M}(G)$ and $\tilde{\mathcal{A}}_{\M'}(G)$ agree for a sectorial region $G$.
\end{rema}

%\section{Results on quasianalyticity in ultraholomorphic classes}

\section{$\M-$summability}\label{sectMsummability}

We are firstly interested in characterizing general ultraholomorphic quasianalytic classes.
%, i.e., those in which the asymptotic Borel map is injective, what will be essential for a good definition of summability. %First, quasianalytic Carleman classes are defined.
%
%\begin{defi}
%Let $S$ be a sector and $\M=(M_{p})_{p\in\N_{0}}$ be a sequence of positive numbers. We say that $\mathcal{A}_{\M}(S)$ is  \textit{quasianalytic} if the conditions:
%\begin{itemize}
%\item[(i)] $f\in\mathcal{A}_{\M}(S)$, and
%\item[(ii)] $\tilde{\mathcal{B}}(f)$ is the null sequence,
%\end{itemize}
%together imply that $f$ is the null function in $S$.
%\end{defi}\parn
%
Due to a classical result of~B. I. Korenbljum~\cite{korenbljum}, the third author has obtained the following one~\cite{sanz13}.

\begin{theo}\label{teororderM}
For a strongly regular sequence $\M$ with associated function $M(r)$, put
\begin{equation}\label{equaordequasM}
\o(\M):=\frac{1}{\rho[M]}\in(0,\infty).
\end{equation}
Then, $\pi\omega(\M)$ is the optimal opening for $\M-$quasianalyticity, in the sense that the class $\mathcal{A}_{\M}(S)$ is (respectively, is not) quasianalytic whenever the opening of $S$ exceeds (resp. is less than) this quantity.
\end{theo}

\begin{rema}\label{remaomegamonotona}
If the strongly regular sequences $\M=(M_n)_{n\in\N_0}$ and $\M'=(M'_n)_{n\in\N_0}$ are such that $M_n\le M'_n$ for every $n$, then $\mathcal{A}_{\bM}(S)\subset\mathcal{A}_{\bM'}(S)$ for any sector $S$, and so  $\omega(\M)\le \omega(\M')$. Moreover, if $\bM$ and $\bM'$ are equivalent then, by Remark~\ref{remaasymspacequisequ}, we see that $\omega(\M)=\omega(\M')$.
\end{rema}

As an easy consequence, we deduce that
%\begin{theo}[generalized Watson's lemma, partial version]
if the opening of a sectorial region $G$ is greater than $\pi\omega(\M)$, then $\tilde{\mathcal{B}}:\tilde{\mathcal{A}}_{\M}(G)\longrightarrow \Lambda_{\M}$ is injective.
So, we are ready for the introduction of a new concept of summability of formal power series in a direction.

\begin{defi}\label{defisumable}
Let $d\in\R$.
We say $\hat{f}=\sum_{n\ge 0}\displaystyle\frac{f_n}{n!}z^n$ is \textit{$\M$-summable in direction $d$} if there exist a sectorial region $G=G(d,\ga)$, with $\ga>\o(\M)$, and
a function $f\in\tilde{\mathcal{A}}_{\M}(G)$ such that $f\sim_{\M}\hat{f}$.
\end{defi}

In this case, by virtue of Proposition~\ref{propcotaderidesaasin}(ii) we have $(f_n)_{n\in\N_0}\in\Lambda_{\M}$. According to Theorem~\ref{teororderM}, $f$ is unique with the property stated, and will be denoted
$$f=\mathcal{S}_{\M,d}\hat{f},\textrm{ the $\M$-sum of $\hat f$ in direction $d$}.$$\parn

Our next aim in this section will be to develop suitable tools in order to recover $f$ from $\hat f$ by means of formal and analytic transforms, in the same vein as in the classical theory for the Gevrey case and the so-called $k-$summability. We will follow the ideas in the theory of general moment summability methods put forward by W. Balser \cite{balserutx}. The case $\omega(\M)<2$ is mainly treated, and indications will be given on how to work in the opposite situation.

\begin{defi}\label{defikernelMsumm}
Let $\M$ be strongly regular with $\omega(\M)<2$.
A pair of complex functions $e,E$ are said to be \textit{kernel functions for $\M-$summability} if:
\begin{itemize}
\item[(\textsc{i})] $e$ is holomorphic in $S_{\omega(\M)}$.
\item[(\textsc{ii})] $z^{-1}e(z)$ is locally uniformly integrable at the origin, i.e., there exists $t_0>0$, and for every $z_0\in S_{\omega(\M)}$ there exists a neighborhood $U$ of $z_0$,
$U\subset S_{\omega(\M)}$, such that the integral $\int_{0}^{t_0}t^{-1}\sup_{z\in U}|e(t/z)|dt$ is finite.

%$z^{-1}e(z)$ is uniformly integrable at the origin, it is to say, for any $t_0>0$ and $\tau_0\in\R$ with $0<\tau_0<\pi\omega(\M)/2$ the integral $\int_{0}^{t_0}t^{-1}\sup_{|\tau|\le\tau_0}|e(te^{i\tau})|dt$ is finite.
\item[(\textsc{iii})] For every $\varepsilon>0$ there exist $c,k>0$ such that
\begin{equation}\label{equacotasnucleoe}
|e(z)|\le ch_{\bM}\left(\frac{k}{|z|}\right)=c\,e^{-M(|z|/k)},\qquad z\in S_{\omega(\M)-\varepsilon},
\end{equation}
where $h_{\bM}$ and $M$ are the functions associated with $\M$, defined in~(\ref{equadefihdeM}) and~(\ref{equadefiMdet}), respectively.
\item[(\textsc{iv})] For $x\in\R$, $x>0$, the values of $e(x)$ are positive real.
\item[(\textsc{v})] If we define the \textit{moment function} associated with $e$,
$$
m_e(\lambda):=\int_{0}^{\infty}t^{\lambda-1}e(t)dt, \qquad\Re(\lambda)\ge 0,
$$
from (\textsc{i})-(\textsc{iv}) we see that $m_e$ is continuous in $\{\Re(\lambda)\ge0\}$, holomorphic in $\{\Re(\lambda)>0\}$, and $m_e(x)>0$ for every $x\ge0$. Then, the function $E$ given by
$$E(z)=\sum_{n=0}^\infty \frac{z^n}{m_e(n)},\qquad z\in\C,$$
is entire, and there exist $C,K>0$ such that
\begin{equation}\label{equacotasE}
|E(z)|\le\displaystyle \frac{C}{h_{\M}(K/|z|)}=Ce^{M(|z|/K)},\quad z\in\C.
\end{equation}
\item[(\textsc{vi})] $z^{-1}E(1/z)$ is locally uniformly integrable at the origin in the sector $S(\pi,2-\o(\M))$, in the sense that there exists $t_0>0$, and for every $z_0\in S(\pi,2-\o(\M))$ there exists a neighborhood $U$ of $z_0$,
$U\subset S(\pi,2-\o(\M))$, such that the integral $\int_{0}^{t_0}t^{-1}\sup_{z\in U}|E(z/t)|dt$ is finite.
\end{itemize}
%LA DE BEDLEWO\par
%Let $\M$ be strongly regular with $\omega(\M)<2$.
%A pair of complex functions $e,E$ are kernel functions for $\M$-summability if:
%\begin{itemize}
%\item[1.] $e$ is holomorphic in $S_{\omega(\M)}$.
%\item[2.] $z^{-1}e(z)$ is integrable at the origin.
%%, it is to say, for any $t_0>0$ and $\tau\in\R$ with $|\tau|<\omega(\M)\pi/2$ the integral $\int_{0}^{t_0}t^{-1}|e(te^{i\tau})|dt$ is finite.
%\item[3.] For every $\varepsilon>0$ there exist $C,K>0$ such that
%\begin{equation*}
%|e(z)|\le Ch_{\bM}\left(\frac{K}{|z|}\right)=C\exp(-M(|z|/K)),\qquad z\in S_{\omega(\M)-\varepsilon}.
%\end{equation*}
%\item[4.] For $x\in\R$, $x>0$, the values of $e(x)$ are positive real.
%\item[5.] If we define $m(\lambda):=\int_{0}^{\infty}t^{\lambda-1}e(t)dt$, $\Re(\lambda)\ge 0$, the function $E$ given by
%      $E(z)=\sum_{n=0}^\infty \frac{z^n}{m(n)}$, $z\in\C$,
%is entire, and there exist $c,k>0$ such that for every $z\in\C$,
%$$|E(z)|\le\displaystyle \frac{c}{h_{\M}(k/|z|)}=c\exp(M(|z|/k)).$$
%\item [6.] $z^{-1}E(1/z)$ is integrable at the origin in the sector $S(\pi,2-\o(\M))$.
%\end{itemize}
\end{defi}

\begin{rema}\label{remadefinucleos}
\begin{itemize}
\item[(i)] The existence of such kernels may be deduced, as we will show in Section~\ref{sectkernelsfromproxorder}, by taking into account the construction of flat functions in $\tilde{\mathcal{A}}_{\bM}(S_{\omega(\bM)})$ accomplished by the third author in~\cite{sanz13}, whenever the function $M(r)$ associated with $\bM$ is such that $d(r)=\log(M(r))/\log(r)$ is a proximate order in the sense of E. Lindel\"of
(see, for example, the book by B. Ya. Levin~\cite{Levin}). This is the case for all the strongly regular sequences appearing in applications, as it has also been shown in~\cite{sanz13}.
\item[(ii)] According to Definition~\ref{defikernelMsumm}(\textsc{v}), the knowledge of $e$ is enough to determine the pair of kernel functions. So, in the sequel we will frequently omit the function $E$ in our statements.
\item[(iii)] In case $\omega(\M)\ge 2$, condition (\textsc{vi}) in Definition~\ref{defikernelMsumm} does not make sense. However, we note that for a positive real number $s>0$ the sequence of $1/s$-powers $\M^{(1/s)}:=(M_n^{1/s})_{n\in\N_0}$ is also strongly regular (see Lemma 1.3.4 in~\cite{thilliez}) and, as it is easy to check, \begin{equation}\label{equahMpotencia}
    h_{\M^{(1/s)}}(t)=\big(h_{\M}(t^{s})\big)^{1/s}, \quad t\ge 0,
    \end{equation}
    and $\omega(\M^{(1/s)})=\omega(\M)/s$. So, following the ideas of Section 5.6 in~\cite{balserutx}, we will say that a complex function $e$ is a kernel for $\M-$summability if there exist $s>0$ with $\omega(\M)/s< 2$, and a kernel $\tilde{e}:S_{\omega(\M)/s}\to\C$ for $\M^{(1/s)}-$summability such that
    $$
    e(z)=\tilde{e}(z^{1/s})/s,\qquad z\in S_{\omega(\M)}.
    $$
    If one defines the moment function $m_{e}$ as before, it is plain to see that $m_e(\lambda)=m_{\tilde{e}}(s\lambda)$, $\Re(\lambda)\ge 0$. The properties verified by $\tilde{e}$ and $m_{\tilde{e}}$ are easily translated into similar ones for~$e$, but in this case the function
    $$
    E(z)=\sum_{n=0}^\infty \frac{z^n}{m_e(n)}=\sum_{n=0}^\infty \frac{z^n}{m_{\tilde{e}}(sn)}
    $$
    does not have the same properties as before, and one rather pays attention to the kernel associated with $\tilde{e}$,
    \begin{equation}\label{equaEtildeomegamayor2}
    \tilde{E}(z)=\sum_{n=0}^\infty \frac{z^n}{m_{\tilde{e}}(n)}=\sum_{n=0}^\infty \frac{z^n}{m_e(n/s)},
    \end{equation}
    which will behave as indicated in (\textsc{v}) and (\textsc{vi}) of Definition~\ref{defikernelMsumm} for such a kernel for $\M^{(1/s)}-$summability.

    It is worth remarking that, once such an $s$ as in the definition exists, one easily checks that for any real number $t>\omega(\M)/2$ a kernel $\hat e$ for $\M^{(1/t)}-$summability exists with $e(z)=\hat{e}(z^{1/t})/t$.
\end{itemize}
\end{rema}

\begin{defi}
Let $e$ be a kernel for $\M-$summability and $m_e$ its associated moment function.
The sequence of positive real numbers $\mathfrak{m}_e=(m_e(p))_{p\in\N_0}$ is known as the \textit{sequence of moments} associated with $e$.
\end{defi}

The following result is a consequence of the estimates, for the kernels $e$ and $E$, appearing in  (\ref{equacotasnucleoe}) and (\ref{equacotasE}) respectively. We omit its proof, since it may be easily adapted from the proof of Proposition 5.7 in~\cite{sanz13}.

\begin{prop}\label{mequivm}
Let $\M=(M_p)_{p\in\N_0}$ be a strongly regular sequence, $e$ a kernel function for $\M-$summability, and $\mathfrak{m}_e=(m_e(p))_{p\in\N_0}$ the sequence of moments associated with $e$. Then $\M$ and $\mathfrak{m}_e$ are equivalent.
\end{prop}

%\begin{rema}
%It may be worth indicating that H. Komatsu
%\end{rema}

\begin{rema}\label{remamomentsarestronglyregular}
\begin{itemize}
\item[(i)] In the Gevrey case of order $\a>0$, $\bM_{\a}=(p!^{\a})_{p\in\N_0}$, it is usual to choose the kernel
$$
e_{\a}(z)=\frac{1}{\a}z^{1/\a}\exp(-z^{1/\a}),\qquad z\in S_{\a}.
$$
Then we obtain that $m_{e_\a}(\lambda)=\Gamma(1+\a \lambda)$ for $\Re(\lambda)\ge 0$. Of course, the sequences $\bM_{\a}$ and $\mathfrak{m}_{e_\a}=(m_{\a}(p))_{p\in\N_0}$ are equivalent.
\item[(ii)] Indeed, for any kernel $e$ for $\M-$summability one may prove that the sequence of moments $\mathfrak{m}_e=(m_e(p))_{p\in\N_0}$ is also strongly regular: Firstly, up to multiplication by a constant scaling factor, one may always suppose that $m_e(0)=1$. Property $(\alpha_0)$ is a consequence of H\"older's inequality, since for every $p\in\N$ one has
\begin{align*}
m_e(p)^2&=\|t^{p-1}e(t)\|_1^2\le \|(t^{p-2}e(t))^{1/2}\|_2^2 \|(t^{p}e(t))^{1/2}\|_2^2=
m_e(p-1)m_e(p+1)
\end{align*}
(where $\|\cdot\|_1$ and $\|\cdot\|_2$ are the standard $L^1$ and $L^2$ norms).
Condition $(\mu)$ for $\mathfrak{m}_e$ is deduced from the equivalence of $\M$ and $\mathfrak{m}_e$ (Proposition~\ref{mequivm}) and condition $(\mu)$ for $\M$: for every $\ell,p\in\N_0$,
\begin{align*}
m_e(p+\ell)&\le A^{p+\ell}M_{p+\ell}\le A^{p+\ell}B^{p+\ell}M_{p}M_{\ell}\le A^{p+\ell}B^{p+\ell}C^{p}C^{\ell}m_e(p)m_e(\ell),
\end{align*}
for suitable positive constants $A,B,C$. Finally, one can find in the work of H.-J. Petzsche~\cite[Corollary~3.2]{pet} that condition $(\gamma_1)$ remains invariant under equivalence of sequences.\par
Bearing this fact in mind, in Definition~\ref{defikernelMsumm} one could depart not from a strongly regular sequence $\M$, but from a kernel $e$, initially defined and positive in direction $d=0$, whose moment function $m_e(\lambda)$ is supposed to be well-defined for $\lambda\ge 0$, and such that the sequence $\mathfrak{m}_e$ is strongly regular. This allows one to consider the constant $\omega(\mathfrak{m}_e)$, which would equal $\omega(\M)$ according to Proposition~\ref{mequivm} and Remark~\ref{remaomegamonotona}, and also the function $h_{\mathfrak{m}_e}$, in terms of which one may rephrase all the items in Definition~\ref{defikernelMsumm}, specially the estimates in~(\ref{equacotasnucleoe}) and~(\ref{equacotasE}), with exactly the same meaning, according to the relationship between $h_{\mathfrak{m}_e}$ and $h_{\M}$ indicated in~(\ref{equahdeMequi}). This insight will be fruitful in Section~\ref{sectmomentPDE}, when dealing with so called moment-partial differential equations.
\end{itemize}
\end{rema}

The next definition resembles that of functions of exponential growth, playing a fundamental role when dealing with Laplace and Borel transforms in $k-$summability for Gevrey classes.
For convenience, we will say a holomorphic function $f$ in a sector $S$ is {\it continuous at the origin} if $\lim_{z\to 0,\ z\in T}f(z)$ exists for every $T\ll S$.

\begin{defi}\label{defiMgrowth}
Let $\bM=(M_{p})_{p\in\N_0}$ be a sequence of positive real numbers verifying ($\a_0$) and ($\gamma_1$), and consider an unbounded sector $S$ in $\mathcal{R}$.
The set $\mathcal{O}^{\bM}(S)$ consists of the holomorphic functions $f$ in $S$, continuous at 0 and having $\M-$growth in $S$, i.e. such that for every unbounded proper subsector $T$ of $S$ there exist $r,c,k>0$ such that for every $z\in T$ with $|z|\ge r$ one has
\begin{equation}\label{equagrowthhM}
|f(z)|\le\frac{c}{h_{\bM}(k/|z|)}.
\end{equation}
\end{defi}
\begin{rema}
Since continuity at 0 has been asked for, $f\in\mathcal{O}^{\bM}(S)$ implies that for every $T\prec S$ there exist $c,k>0$ such that for every $z\in T$ one has (\ref{equagrowthhM}).
\end{rema}

\noindent We are ready for the introduction of the $\bM-$Laplace transform.

Given a kernel $e$ for $\M-$summability, a sector $S=S(d,\alpha)$ and $f\in\mathcal{O}^{\bM}(S)$, for any direction $\tau$ in $S$ we define the operator $T_{e,\tau}$  sending $f$ to its \textit{$e$-Laplace transform in direction $\tau$}, defined as
\begin{equation}\label{equadefitranLapl}
(T_{e,\tau}f)(z):=\int_0^{\infty(\tau)}e(u/z)f(u)\frac{du}{u},\quad
|\arg(z)-\tau|<\o(\M)\pi/2,\ |z|\textrm{ small enough},
\end{equation}
where the integral is taken along the half-line parameterized by $t\in(0,\infty)\mapsto te^{i\tau}$.
We have the following result.

\begin{prop}\label{prophollap}
For a sector $S=S(d,\alpha)$ and $f\in\mathcal{O}^{\bM}(S)$, the family $\{T_{e,\tau}f\}_{\tau\textrm{\,in\,}S}$ defines a holomorphic function $T_{e}f$ in a
sectorial region $G(d,\a+\o(\M))$.
\end{prop}

\begin{proof1}
Let $\tau\in\R$ be a direction in $S$, i.e., such that $|\tau-d|<\a\pi/2$. We will show that for every $\beta$ with $0<\beta<\o(\M)$, there exists $r=r(f,\tau,\beta)>0$ such that $T_{e,\tau}f$ is holomorphic in $S(\tau,\beta,r)$. Hence,
$T_{e,\tau}f$ will be holomorphic in $G_{\tau}:=\cup_{0<\beta<\o(\M)}S(\tau,\beta,r)$, which is a sectorial region $G_{\tau}=G(\tau,\o(\M))$.\par
For every $u,z\in\mathcal{R}$ with $\mathrm{arg}(u)=\tau$ and $|\mathrm{arg}(z)-\tau|<\o(\M)\pi/2$ we have that $u/z\in S_{\o(\M)}$, so that the expression under the integral sign in (\ref{equadefitranLapl}) makes sense.
We fix $a>0$, and write
$$
\int_0^{\infty(\tau)}e(u/z)f(u)\frac{du}{u}=\int_0^{ae^{i\tau}}e(u/z)f(u)\frac{du}{u}+\int_{ae^{i\tau}}^{\infty(\tau)}e(u/z)f(u)\frac{du}{u}.
$$
Since $f$ is continuous at the origin, and because of Definition~\ref{defikernelMsumm}(\textsc{ii}), it is straightforward to apply Leibnitz's rule for parametric integrals and deduce that
the first integral in the right-hand side defines a holomorphic function in $S(\tau,\o(\M))$. Regarding the second integral, for $u$ as before and by Definition~\ref{defiMgrowth} there exist $c_1,k_1>0$ such that
$$
|f(u)|\le c_1\big(h_{\M}(k_1/|u|)\big)^{-1}.
$$
Also, for $z$ such that $|\mathrm{arg}(z)-\tau|<\b\pi/2$ we have that $u/z\in S_{\b}$, and the property \ref{defikernelMsumm}(\textsc{iii}) provides us with constants $c_2,k_2>0$ such that
$$
|e(u/z)|\le c_2h_{\M}(k_2|z|/|u|),
$$
so that
$$
\big|\frac{1}{u}e(u/z)f(u)\big|\le \frac{c_1c_2}{|u|}\frac{h_{\bM}(k_2|z|/|u|)}{h_{\bM}(k_1/|u|)}.
$$
Let $\rho(2)>0$ be the constant appearing in (\ref{e120}) for $s=2$, and consider $r:=k_1/(\rho(2)k_2)>0$. For any $z\in S(\tau,\beta,r)$ we have that $\rho(2)k_2|z|<k_1$, and from (\ref{e120}) and the monotonicity of $h_{\M}$ we deduce that
\begin{equation*}%\label{equacotacocientehMs}
\big|\frac{1}{u}e(u/z)f(u)\big|\le \frac{c_1c_2}{|u|}\frac{h_{\bM}^2(\rho(2)k_2|z|/|u|)}{h_{\bM}(k_1/|u|)}\le \frac{c_1c_2}{|u|}h_{\bM}(k_1/|u|).
\end{equation*}
By the very definition of $h_{\M}$ we have that $h_{\M}(k_1/|u|)\le M_1k_1/|u|$, so the right-hand side of the last inequality is an integrable function of $|u|$ in $(a,\infty)$, and again Leibnitz's rule allows us to conclude the desired analyticity for the second integral.\par
Let $\sigma\in\R$ with $|\sigma-d|<\alpha\pi/2$. The map $T_{e,\sigma}f$ is a holomorphic function in a sectorial region $G_{\sigma}=G(\sigma,\o(\M))$ which will overlap with $G_{\tau}$ whenever $\tau$ and $\sigma$  are close enough. Since we know that
\begin{equation*}
\lim_{t\to\infty}th_{\bM}(1/t)=0,
\end{equation*}
by Cauchy's residue theorem we easily deduce that $T_{e,\tau}f(z)\equiv T_{e,\sigma}f(z)$ whenever both maps are defined. Thus the family $\{T_{e,\tau}f\}_{\tau\textrm{\,in\,}S}$ defines a holomorphic function $T_{e}f$ in the union of the sectorial regions $G_{\tau}$, which is indeed again a sectorial region $G=G(d,\a+\o(\M))$.
\end{proof1}

We now define the generalized Borel transforms.\par

Suppose $\o(\M)<2$, and let $G=G(d,\a)$ be a sectorial region with $\a>\o(\M)$, and $f:G\to \C$ be holomorphic in $G$ and continuous at 0.
For $\tau\in\R$ such that $|\tau-d|<(\a-\o(\M))\pi/2$ we may consider a path
$\delta_{\o(\M)}(\tau)$ in $G$ like the ones used in the classical Borel transform, consisting of a segment from the origin to a point $z_0$ with $\arg(z_0)=\tau+\o(\M)(\pi+\varepsilon)/2$ (for some
suitably small $\varepsilon\in(0,\pi)$), then the circular arc $|z|=|z_0|$ from $z_0$ to
the point $z_1$ on the ray $\arg(z)=\tau-\o(\M)(\pi+\varepsilon)/2$ (traversed clockwise), and
finally the segment from $z_1$ to the origin.

Given kernels $e,E$ for $\M-$summability, we define the operator $T^{-}_{e,\tau}$ sending $f$ to its \textit{$e$-Borel transform in direction $\tau$}, defined as
$$
(T^{-}_{e,\tau}f)(u):=\frac{-1}{2\pi i}\int_{\delta_{\o(\M)}(\tau)}E(u/z)f(z)\frac{dz}{z},\quad
u\in S(\tau,\varepsilon_0), \quad \varepsilon_0\textrm{ small enough}.
$$

In case $\o(\M)\ge 2$, choose $s>0$ and a kernel $\tilde{e}$ for $\M^{(1/s)}-$summability as in Remark~\ref{remadefinucleos}(iii), and let $T^{-}_{\tilde{e},\tau}$ be defined as before, where the kernel under the integral sign is the function $\tilde{E}$ given in (\ref{equaEtildeomegamayor2}). Then, if $\phi_s$ is the operator sending a function $f$ to the function $f(z^s)$, we define $T^{-}_{e,\tau}$
by the identity
\begin{equation}\label{equadefiBoreltransomegamayor2}
\phi_s\circ T^{-}_{e,\tau}=T^{-}_{\tilde{e},\tau}\circ \phi_s,
\end{equation}
in the same way as in~\cite[p.\ 90]{balserutx}.

\begin{prop}\label{propholBor}
For $G=G(d,\a)$ and $f:G\to \C$ as above, the family
$$
\{T^{-}_{e,\tau}f\}_{\tau},
$$
where $\tau$ is a real number such that $|\tau-d|<(\a-\o(\M))\pi/2$, defines a holomorphic function $T^{-}_{e}f$ in the sector $S=S(d,\a-\o(\M))$. Moreover, $T^{-}_{e}f$ is of $\M$-growth in $S$.
\end{prop}

\begin{proof1}
Due to the identity (\ref{equadefiBoreltransomegamayor2}), it is clearly sufficient to prove our claim in the case $\o(\M)<2$.
Since $f$ is holomorphic in $G$ and continuous at 0, for every $\tau\in\R$ such that $|\tau-d|<(\a-\o(\M))\pi/2$ the condition in Definition \ref{defikernelMsumm}(\textsc{vi}) implies that $T^{-}_{e,\tau}f$ is holomorphic
in the sector $S(\tau,\varepsilon\omega(\M)/\pi)$, where $\varepsilon>0$ is the one entering in the definition of $\delta_{\o(\M)}(\tau)$, and it is small enough so that $u/z$ stays in the sector $S(\pi,2-\omega(\M))$ as $u\in S(\tau,\varepsilon\omega(\M)/\pi)$ and $z$ runs over the two segments in $\delta_{\o(\M)}(\tau)$. Cauchy's theorem easily implies that the family $\{T^{-}_{e,\tau}f\}$, when $\tau\in\R$ and $|\tau-d|<(\a-\o(\M))\pi/2$, defines a holomorphic function $T^{-}_{e}f$ in the sector $S=S(d,\a-\o(\M))$. Let us finally check that $T^{-}_{e}f$ is of $\M$-growth in $S$. By compactness, it suffices to work on a proper unbounded subsector $T$ of $S(\tau,\varepsilon\omega(\M)/\pi)$, for $\tau$ and $\varepsilon$ as before. Put $\delta_{\o(\M)}(\tau)=\delta_1+\delta_2+\delta_3$, where $\delta_1$ and $\delta_3$ are the aforementioned segments in directions, say, $\theta_1$ and $\theta_3$, and $\delta_2$ is the circular arc with radius $r_2>0$.
As $f$ is continuous at the origin, there exists $M>0$ such that $|f(z)|\le M$ for every $z$ in the trace of
$\delta_{\o(\M)}(\tau)$. So, for every $u\in T$ and $j=1,3$ we have
\begin{align*}%\label{equacotaBoreldelta13}
%|T^{-}_{e}f(u)|&=|T^{-}_{e,\tau}f(u)|=
\Big|\frac{-1}{2\pi i}\int_{\delta_j}E(u/z)f(z)\frac{dz}{z}\Big|
&\le\frac{M}{2\pi}\int_0^{r_2}\frac{1}{t}|E(ue^{-i\theta_j}/t)|dt=
\frac{M}{2\pi}\int_0^{r_2/|u|}\frac{1}{s}|E(e^{i(\arg(u)-\theta_j)}/s)|ds,
\end{align*}
after the change of variable $t=|u|s$.
According to the condition in Definition \ref{defikernelMsumm}(\textsc{vi}), these expressions uniformly tend to 0 as $u$ tends to infinity in $T$. On the other hand, the estimates for $E$ in (\ref{equacotasE}) allow us to write
\begin{align*}%\label{equacotaBoreldelta2}
%|T^{-}_{e}f(u)|&=|T^{-}_{e,\tau}f(u)|=
\Big|\frac{-1}{2\pi i}\int_{\delta_2}E(u/z)f(z)\frac{dz}{z}\Big|
&\le\frac{M(\theta_1-\theta_3)}{2\pi}\max_{|z|=r_2}|E(u/z)|
\le \frac{CM(\theta_1-\theta_3)}{2\pi h_{\M}(Kr_2/|u|)}.
\end{align*}
So, this integral has the desired growth and the conclusion follows.
% from (\ref{equacotaBoreldelta13}) and (\ref{equacotaBoreldelta2}).
\end{proof1}

In the next paragraph we follow the same ideas in \cite[p.\ 87-88]{balserutx} in order to justify the forthcoming definition of formal Laplace and Borel transforms.\par
Let $\bM=(M_{p})_{p\in\N_0}$ be a strongly regular sequence, $e$ a kernel for $\M-$summability and $S=S_\alpha$. It is clear that for every $\lambda\in\C$ with $\Re(\lambda)\ge0$, the function $f_{\lambda}(z)=z^{\lambda}$ belongs to the space $\mathcal{O}^{\bM}(S)$. From Proposition~\ref{prophollap}, one can define $T_ef_{\lambda}(z)$ for every $z$ in an appropriate sectorial region $G$. Moreover, for $z\in G$ and an adequate choice of $\tau\in\R$ one has
$$T_ef_{\lambda}(z)= \int_{0}^{\infty(\tau)}e\left(\frac{u}{z}\right)u^{\lambda-1}du.$$
In particular, for $z\in\mathcal{R}$ with $\hbox{arg}(z)=\tau$, the change of variable $u/z=t$ turns the preceding integral into
\begin{equation}\label{equaLapltransmonom}
T_ef_{\lambda}(z)=\int_{0}^{\infty}e(t)z^{\lambda-1}t^{\lambda-1}zdt= m_e(\lambda)z^{\lambda}.
\end{equation}

Next, we recall the following result by H. Komatsu, which was useful in the proof of Proposition~\ref{mequivm}.

\begin{prop}[\cite{komatsu},\ Proposition\ 4.5]\label{propKomatsu}
Let $M(r)$ be the function associated with a sequence $\M$ verifying $(\a_0)$. Given an entire function $F(z)=\sum_{n=0}^\infty a_nz^n$, $z\in\C$, the following statements are equivalent:
\begin{itemize}
\item[(i)] $F$ is of $\M-$growth.
\item[(ii)] There exist $c,k>0$ such that for every $n\in\N_0$, $|a_n|\le ck^n/M_n$.
\end{itemize}
\end{prop}

Taking this characterization into account, one may justify termwise integration to obtain that, for such a function $F$,
$$
T_eF(z)=\sum_{n=0}^\infty a_nm_e(n)z^n
$$
whenever  $|z|$ is small enough. In case $\o(\M)<2$, and particularizing this result for $F=E$ (the kernel function corresponding to $e$), we deduce that for every $z\neq 0$ and $w\neq 0$ such that $|z/w|<1$ one has
\begin{equation}\label{equaLapltransE}
\frac{w}{w-z}=\int_0^{\infty(\tau)}e(u/z)E(u/w)\frac{du}{u},
\end{equation}
a formula which remains valid as long as both sides are defined.
Suppose now that $f$ is holomorphic in a sectorial region $G(d,\a)$, with $\a>\omega(\M)$, and continuous at the origin. By Propositions \ref{propholBor} and \ref{prophollap}, $T_eT_e^{-}f$ is well defined; a change in the order of integration and the use of (\ref{equaLapltransE}) prove that
\begin{equation}\label{equaLaplBorelident}
T_eT_e^{-}f=f.
\end{equation}
Finally, since $f_{\lambda}$ (defined above) is continuous at the origin, it makes sense to compute
$$
T^{-}_ef_{\lambda}(u)=\frac{-1}{2\pi i}\int_{\delta_{\o(\M)}(\tau)}E(u/z)z^{\lambda-1}dz.
$$
Putting $u/z=t$, the integral is changed into
$$
T^{-}_ef_{\lambda}(u)=\frac{u^{\lambda}}{2\pi i}\int_{\gamma_z}E(t)t^{-\lambda-1}dt,
$$
for a corresponding path $\gamma_z$. However, Cauchy's theorem allows one to choose one and the same path of integration as long as $z$ runs in a suitably small disk, and we deduce, by the identity principle, that $T^{-}_ef_{\lambda}(u)$ is a constant multiple of $u^{\lambda}$ in $S(d,\a-\omega(\M))$. According to (\ref{equaLapltransmonom}) and (\ref{equaLaplBorelident}), we conclude that
\begin{equation}\label{equaBoretransmonom}
T^{-}_ef_{\lambda}(u)=\frac{u^{\lambda}}{m_e(\lambda)}.
\end{equation}
Observe that, taking into account (\ref{equadefiBoreltransomegamayor2}), the same will be true if $\o(\M)\ge 2$. Therefore, it is adequate to make the following definitions.

\begin{defi}
Given a strongly regular sequence $\bM$ and a kernel for $\M-$summability $e$, the formal $e-$Laplace transform $\hat{T}_{e}:\C[[z]]\to\C[[z]]$ is given by
$$\hat{T}_{e}\big(\sum_{p=0}^{\infty}a_{p}z^{p}\big):=\sum_{p=0}^{\infty}m_e(p)a_{p}z^{p},\qquad \sum_{p=0}^{\infty}a_{p}z^{p}\in\C[[z]].$$
Accordingly, we define the formal $e-$Borel transform $\hat{T}_e^{-}:\C[[z]]\to\C[[z]]$ by
$$\hat{T}_e^{-}\big(\sum_{p=0}^{\infty}a_{p}z^{p}\big):=\sum_{p=0}^{\infty}\frac{a_{p}}{m_e(p)}z^{p},\qquad \sum_{p=0}^{\infty}a_{p}z^{p}\in\C[[z]].$$
\end{defi}
The operators $\hat{T}_e$ and $\hat{T}_e^{-}$ are inverse to each other.

%For $\lambda\in\C$ with $\Re(\lambda)>0$, one has
%$$
%T_e(z^{\lambda})=m(\lambda)z^{\lambda},\qquad T^{-}_e(z^{\lambda})=\frac{1}{m(\lambda)}z^{\lambda}.
%$$
%
%The \textit{formal $e$-Laplace and Borel operators, $\hat{T}_e$ and $\hat{T}_e^{-}$}, are accordingly defined in $\C[[z]]$ by linearity.\parn

The next result lets us know how these analytic and formal transforms interact with general asymptotic expansions. Given two sequences of positive real numbers $\M=(M_n)_{n\in\N_0}$ and $\M'=(M'_n)_{n\in\N_0}$, we consider the sequences $\M\cdot\M'=(M_nM'_n)_{n\in\N_0}$ and $\M'/\M=(M'_n/M_n)_{n\in\N_0}$. We note that if both sequences are strongly regular, $\M\cdot\M'$ is again strongly regular (see Proposition~\ref{propprodfuertregu}), but $\M'/\M$ might not be.

\begin{rema}
One may easily prove (see Remark 5.8 in \cite{sanz13}) that given $K>0$, there exist $C,D>0$ such that for every $p\in\N$ one has
\begin{equation}\label{equacotasintegrhM}
\int_{0}^{\infty}t^{p-1}h_{\bM}(K/t)dt\le CD^pM_p.
\end{equation}
%This will be useful in what follows.
\end{rema}

\begin{theo}\label{teorrelacdesartransfBL}
Suppose $\M$ is strongly regular and $e$ is a kernel for $\M-$summability. For any sequence $\M'$ of positive real numbers the following hold:
\begin{itemize}
\item[(i)] If $f\in\mathcal{O}^{\bM}(S(d,\a))$
and $f\sim_{\M'}\hat{f}$, then $T_ef\sim_{\M\cdot\M'}\hat{T}_e\hat{f}$ in a
sectorial region $G(d,\a+\o(\M))$.
\item[(ii)] If $f\sim_{\M'}\hat{f}$ in a sectorial region $G(d,\a)$ with $\a>\o(\M)$, then
    $T^{-}_ef\sim_{\M'/\M}{\hat{T}}^{-}_e\hat{f}$ in the sector $S(d,\a-\o(\M))$.
\end{itemize}
\end{theo}

\begin{proof1}
(i) From Proposition~\ref{prophollap} we know that $g:=T_ef\in\mathcal{O}(G(d,\a+\o(\M)))$ for a sectorial region $G=G(d,\a+\o(\M))$. Put $\hat{f}=\sum_{n=0}^\infty f_nu^n$. Given $\delta\in(0,\a)$, there exist $c,k>0$ such that for every $u\in S(d,\delta,1)$ and every $n\in\N$ one has
\begin{equation*}
\big|f(u)-\sum_{k=0}^{n-1} f_ku^k\big|\le ck^nM_n'|u|^n.
\end{equation*}
Then, we deduce that $|f_n|\le ck^nM_n'$ for every $n$. Also, since $f$ is of $\M-$growth, for every $u\in S(d,\delta)$ we have, by suitably enlarging the constants,
\begin{equation}\label{equaestimaproxordennMgrowth}
\big|f(u)-\sum_{k=0}^{n-1} f_ku^k\big|\le \frac{c_1k_1^nM_n'|u|^n}{h_{\M}(k_1/|u|)}.
\end{equation}
Observe that, by (\ref{equaLapltransmonom}), for every $z\in G$ we have
$$
g(z)-\sum_{k=0}^{n-1} m_e(k)f_kz^k=T_e\big(f(u)-\sum_{k=0}^{n-1} f_ku^k\big)(z).
$$
So, given $\tau\in\R$ with $|\tau-d|<\a\pi/2$ and $z\in S(\tau,\b)$ with $\b\in(0,\omega(\M))$ and $|z|$ small enough, we have
$$
g(z)-\sum_{k=0}^{n-1} m_e(k)f_kz^k=\int_0^{\infty(\tau)}e(u/z)\big(f(u)-\sum_{k=0}^{n-1} f_ku^k\big)\frac{du}{u}.
$$
Since $u/z\in S_{\b}$, by Definition~\ref{defikernelMsumm}(\textsc{iii}) there exist $c_2,k_2>0$ such that
$$
|e(u/z)|\le c_2h_{\M}(k_2|z|/|u|),
$$
and so, taking into account (\ref{equaestimaproxordennMgrowth}) and (\ref{e120}),
$$
\Big|\frac{e(u/z)}{u}\big(f(u)-\sum_{k=0}^{n-1} f_ku^k\big)\Big|\le c_1c_2k_1^nM_n'|u|^{n-1}\frac{h_{\bM}(k_2|z|/|u|)}{h_{\bM}(k_1/|u|)}\le
c_1c_2k_1^nM_n'|u|^{n-1}\frac{h_{\bM}^2(\rho(2)k_2|z|/|u|)}{h_{\bM}(k_1/|u|)}.
$$
For $z\in S(\tau,\beta,k_1/(\rho(2)k_2))$ we have $\rho(2)k_2|z|/|u|<k_1/|u|$ and,
$h_{\M}$ being increasing, we obtain that
$$
\Big|\frac{e(u/z)}{u}\big(f(u)-\sum_{k=0}^{n-1} f_ku^k\big)\Big|\le c_1c_2k_1^nM_n'|u|^{n-1}h_{\bM}(\rho(2)k_2|z|/|u|),
$$
what implies that
$$
\Big|g(z)-\sum_{k=0}^{n-1} m_e(k)f_kz^k\Big|\le
c_1c_2k_1^nM_n'\int_{0}^{\infty}s^{n-1}h_{\bM}(\rho(2)k_2|z|/s)ds.
$$
Making the change of variable $s=|z|t$ and applying~(\ref{equacotasintegrhM}) leads to the conclusion.\par\noindent
(ii) From Proposition~\ref{propholBor}, $g:=T^{-}_ef$ belongs to $\mathcal{O}(S(d,\a-\o(\M)))$.
As in the proof of that Proposition, we limit ourselves to the case $\o(\M)<2$. It suffices to obtain estimates on a bounded subsector $T=S(\tau,\varepsilon\omega(\M)/\pi,\rho)\ll S(d,\a-\omega(\M))$, for $\tau$ and $\varepsilon$ entering in the definition of the path $\delta_{\o(\M)}(\tau)=\delta_1+\delta_2+\delta_3$ within $G(d,\a)$ ($\delta_1$ and $\delta_3$ are segments in directions $\theta_1$ and $\theta_3$, respectively, and $\delta_2$ is a circular arc with radius $r>0$).
If $\hat{f}=\sum_{n=0}^\infty f_nu^n$, there exist $c,k>0$ such that for every $z$ in the trace of $\delta_{\o(\M)}(\tau)$ and every $n\in\N$ one has
\begin{equation}\label{equacotasdesaMprima}
\big|f(z)-\sum_{k=0}^{n-1} f_kz^k\big|\le ck^nM_n'|z|^n.
\end{equation}
%Then, we deduce that $|f_n|\le ck^nM_n'$ for every $n$. Also, since $f$ is of $\M-$growth, for every $u\in S(d,\delta)$ we have, by suitably enlarging the constants,
%\begin{equation}\label{equaestimaproxordennMgrowth}
%\big|f(u)-\sum_{k=0}^{n-1} f_ku^k\big|\le \frac{c_1k_1^nM_n'|u|^n}{h_{\M}(k_1/|u|)}.
%\end{equation}
By (\ref{equaBoretransmonom}), for every $u\in T$ we have
\begin{align}
g(u)-\sum_{k=0}^{n-1} \frac{f_k}{m_e(k)}u^k&=T^{-}_e\big(f(z)-\sum_{k=0}^{n-1} f_kz^k\big)(u)=
\frac{-1}{2\pi i}\int_{\delta_{\o(\M)}(\tau)}E(u/z)\big(f(z)-\sum_{k=0}^{n-1} f_kz^k\big)\frac{dz}{z}\nonumber\\
&=\frac{-1}{2\pi i}\sum_{j=1}^3\int_{\delta_j}E(u/z)\big(f(z)-\sum_{k=0}^{n-1} f_kz^k\big)\frac{dz}{z}.\label{equaBorelsuma3integ}
\end{align}
By applying (\ref{equacotasdesaMprima}) and (\ref{equacotasE}), we see that
$$
\Big|\int_{\delta_2}E(u/z)\big(f(z)-\sum_{k=0}^{n-1} f_kz^k\big)\frac{dz}{z}\Big|\le
c(\theta_1-\theta_3)k^nM_n'r^{n}\max_{|z|=r}|E(u/z)|
\le \frac{cC(\theta_1-\theta_3)k^nM_n'r^{n}}{h_{\M}(Kr/|u|)}.
$$
So, for $n$ large enough we may choose $r=|u|/(Km_n)$ and, since $h_{\M}(1/m_n)=M_n/m_n^n$, we deduce that
\begin{equation}\label{equacotadelta2}
\Big|\int_{\delta_2}E(u/z)\big(f(z)-\sum_{k=0}^{n-1} f_kz^k\big)\frac{dz}{z}\Big|
\le \frac{cC(\theta_1-\theta_3)k^nM_n'|u|^n}{K^nM_n}.
\end{equation}
On the other hand,
again (\ref{equacotasdesaMprima}) implies that for $j=1,3$,
\begin{align*}
\Big|\int_{\delta_j}E(u/z)\big(f(z)-\sum_{k=0}^{n-1} f_kz^k\big)\frac{dz}{z}\Big|&\le
ck^nM_n'\int_0^{r}t^{n}\frac{|E(ue^{-i\theta_j}/t)|}{t}dt\\
&=ck^nM_n'|u|^n\int_0^{r/|u|}s^{n}\frac{|E(e^{i(\arg(u)-\theta_j)}/s)|}{s}ds,
\end{align*}
after the change of variable $t=|u|s$. The same choice of $r$ as before leads to
\begin{align}
\Big|\int_{\delta_j}E(u/z)\big(f(z)-\sum_{k=0}^{n-1} f_kz^k\big)\frac{dz}{z}\Big|
&\le ck^nM_n'|u|^n\int_0^{1/(Km_n)}s^n\frac{|E(e^{i(\arg(u)-\theta_j)}/s)|}{s}ds\nonumber\\
&\le ck^nM_n'|u|^n\frac{1}{K^nm_n^n}\int_0^{1/(Km_n)}\frac{|E(e^{i(\arg(u)-\theta_j)}/s)|}{s}ds.
\label{equacotadeltaj}\end{align}
Since $\lim_{n\to\infty}m_n=\infty$, the last integral admits an upper bound independent of $u$ and $n$ because of condition (\textsc{vi}) in Definition~\ref{defikernelMsumm}. According to
(\ref{equameneequivMene}), (\ref{equaBorelsuma3integ}), (\ref{equacotadelta2}) and (\ref{equacotadeltaj}), the conclusion is reached.
\end{proof1}

We are ready for giving a definition of summability in a direction with respect to a kernel $e$ of $\M$-summability. Let $T_e$ be the corresponding Laplace operator, and recall that $\mathfrak{m}_e$ is strongly regular and equivalent to $\M$, so that, on one hand, $\Lambda_{\M}=\Lambda_{\mathfrak{m}_e}$ and, on the other hand, it makes sense to consider the space $\mathcal{O}^{\mathfrak{m}_e}(S)$ for any unbounded sector $S$ and, moreover,
$\mathcal{O}^{\mathfrak{m}_e}(S)=\mathcal{O}^{\M}(S)$ (see (\ref{equahdeMequi})).\par\noindent

\begin{defi}\label{defisumableporunnucleo}
We say $\hat{f}=\sum_{n\ge 0}\displaystyle\frac{f_n}{n!}z^n$ is \textit{$T_e$-summable in direction $d\in\R$} if:
 \begin{itemize}
 \item[(i)] $(f_n)_{n\in\N_0}\in\Lambda_{\mathfrak{m}_e}$, so that $g:=
\hat{T}_e^{-}\hat{f}=\sum_{n\ge 0}\displaystyle\frac{f_n}{n!m_e(n)}z^n$ converges in a disc, and
\item[(ii)] $g$ admits analytic continuation in a sector $S=S(d,\varepsilon)$ for some $\varepsilon>0$, and $g\in\mathcal{O}^{\mathfrak{m}_e}(S)$.
\end{itemize}
\end{defi}

The next result states the equivalence between $\M-$summability and $T_e-$summability in a direction, and provides a way to recover the $\M-$sum in a direction of a summable power series by means of the formal and analytic transforms previously introduced.

\begin{theo}\label{teorsumableequivTsumable}
Given a strongly regular sequence $\M$, a direction $d$ and a formal power series $\hat f=\sum_{n\ge 0}\displaystyle\frac{f_n}{n!}z^n$, the following are equivalent:
\begin{itemize}
\item[(i)] $\hat f$ is $\M$-summable in direction $d$.
\item[(ii)] For every kernel $e$ of $\M$-summability, $\hat{f}$ is $T_e$-summable in direction $d$.
\item[(iii)] For some kernel $e$ of $\M$-summability, $\hat{f}$ is $T_e$-summable in direction $d$.
\end{itemize}
In case any of the previous holds, we have (after analytic continuation)
$$
\mathcal{S}_{\M,d}\hat{f}=T_e(\hat{T}_e^{-}\hat{f})
$$
for any kernel $e$ of $\M$-summability.
\end{theo}

\begin{proof1}
(i)$\implies$(ii) Let $f=\mathcal{S}_{\M,d}\hat{f}$, the $\M-$sum of $\hat f$ in direction $d$. Then $f\sim_{\M}\hat{f}$ in a sectorial region $G(d,\a)$ with $\a>\omega(\M)$, and moreover $(f_n)_{n\in\N_0}\in\Lambda_{\M}$. If we put $\M'=(1)_{n\in\N_0}$ (the constant sequence whose terms are all equal to 1), item (ii) in Theorem~\ref{teorrelacdesartransfBL} states that
$g:=T^{-}_ef\sim_{\M'}{\hat{T}}^{-}_e\hat{f}$, what implies that ${\hat{T}}^{-}_e\hat{f}$ converges to $g$ in a disk, and $g$ is, by Proposition~\ref{propholBor}, of $\M-$growth in a small unbounded sector around $d$, as we intended to prove.
\par\noindent
(ii)$\implies$(iii) Trivial.\par\noindent
(iii)$\implies$(i) Since $g:=\hat{T}_e^{-}\hat{f}$ converges in a disc and admits analytic continuation in a sector $S=S(d,\varepsilon)$ for some $\varepsilon>0$, we have that
$g\sim_{\M'}\hat{T}_e^{-}\hat{f}$ in $S$ with $\M'=(1)_{n\in\N_0}$. Moreover, $g\in\mathcal{O}^{\mathfrak{m}_e}(S)=\mathcal{O}^{\bM}(S)$, and due to (i) in Theorem~\ref{teorrelacdesartransfBL}, we obtain that the function $f:=T_eg$ is holomorphic in a sectorial region of opening greater than $\pi\omega(\M)$ and $f\sim_{\M}\hat{f}$ there, so we are done.
\end{proof1}

%\begin{prop}
%Let $\hat{f}=\sum_{n\ge 0}\displaystyle\frac{f_n}{n!}z^n$ be $\bM-$summable in direction $d\in\R$.
%Then, there exists a sectorial region
%$G(d,\beta)$, with $\beta>\gamma(\bM)$, and a function $f\in\mathcal{A}_{\bM}(G(d,\beta))$ such that $f\sim_{\bM}\hat{f}$ in $G(d,\beta)$.
%\end{prop}

\begin{rema}\label{remanotassumabM}
\begin{itemize}
\item[(i)] In case $\M=\M_{1/k}$, the summability methods described are just the classical $k$-summability and $T_e$-summability (in a direction) for kernels $e$ of order $k>0$.
%\item[(ii)] In case $\o(\M)\ge 2$, modifications are needed as in W. Balser~\cite{balserutx}.
\item[(ii)] If $\M$ and $\mathbb{M}'$ are equivalent strongly regular sequences, the respective families of kernels of summability coincide, as it is easily deduced from~(\ref{equahdeMequi}), hence the summability methods just introduced for $\M$ and $\M'$ are all the same, as well as the sums provided for every $\M-$ (or equivalently, $\M'-$) summable series in a direction.

    In particular, consider a kernel $e$, its moment function $m_e$ and the strongly regular sequence of moments $\mathfrak{m}_e$, as in Remark~\ref{remamomentsarestronglyregular}(ii). According to Definition~\ref{defikernelMsumm}, Remark~\ref{remadefinucleos}(iii) and Remark~\ref{remamomentsarestronglyregular}(ii), for every $s>0$ one may deduce that $e^{(s)}(z):=e(z^{1/s})/s$ is a kernel for $\mathfrak{m}_e^{(s)}-$summability (recall that $\mathfrak{m}_e^{(s)}=(m_e^s(n))_{n\in\N_0}$) with moment function $m_{e^{(s)}}(\lambda)=m_e(s\lambda)$ and sequence of moments $\mathfrak{m}_{e^{(s)}}=(m_e(sn))_{n\in\N_0}$, and consequently, $\mathfrak{m}_e^{(s)}$ and $\mathfrak{m}_{e^{(s)}}$ are equivalent (Proposition~\ref{mequivm}) and
    \begin{equation}\label{equaomegamomentosescalados}
    \o((m_e(sn))_{n\in\N_0})=\o(\mathfrak{m}_e^{(s)})=s\o(\mathfrak{m}_e)=s\o((m_e(n))_{n\in\N_0}).
    \end{equation}
    Moreover, by (\ref{equahdeMequi}) and (\ref{equahMpotencia}), there exist $A,B>0$ such that for every $t\ge 0$ one has
    \begin{equation}\label{equahdeMese}
    \big(h_{\mathfrak{m}_e}(At)\big)^s=h_{\mathfrak{m}_e^{(s)}}(A^st^s)\le h_{\mathfrak{m}_{e^{(s)}}}(t^s)\le h_{\mathfrak{m}_e^{(s)}}(B^st^s)=\big(h_{\mathfrak{m}_e}(Bt)\big)^s.
    \end{equation}
\end{itemize}
\end{rema}

\section{Kernels of summability from proximate orders}\label{sectkernelsfromproxorder}

In this section we show how one can construct kernels of summability for a strongly regular sequence $\M$ by relying on the notion of analytic proximate orders, appearing in the theory of growth of entire functions and developed, among others,  by E. Lindel\"of, G. Valiron, B. Ja. Levin, A. A. Goldberg, I. V. Ostrosvkii and L. S. Maergoiz (see the references \cite{Valiron,Levin,GoldbergOstrowskii,Maergoiz}).

\begin{defi}[\cite{Valiron}]%\label{defiproxorde}
We say a real function $\ro(r)$, defined on $(c,\infty)$ for some $c\ge 0$, is a \textit{proximate order} if the following hold:
 \begin{enumerate}[(i)]
  \item $\rho(r)$ is continuous and piecewise continuously differentiable in $(c,\infty)$,
  \item $\ro(r) \geq 0$ for every $r>c$,
  \item $\lim_{r \to \infty} \ro(r)=\ro< \infty$,
  \item $\lim_{r  \to \infty} r \ro'(r) \log(r) = 0$.
 \end{enumerate}
\end{defi}

\begin{defi}
Two proximate orders $\rho_1(r)$ and $\rho_2(r)$ are said to be \textit{equivalent} if
$$
\lim_{r\to\infty}\big(\rho_1(r)-\rho_2(r)\big)\log(r)=0.
$$
\end{defi}

\begin{rema}\label{remaordenaproxequiv}
If $\rho_1(r)$ and $\rho_2(r)$ are equivalent and $\lim_{r\to\infty}\rho_1(r)=\rho$, then  $\lim_{r\to\infty}\rho_2(r)=\rho$ and  $\lim_{r\to\infty}r^{\rho_1(r)}/r^{\rho_2(r)}=1$.
\end{rema}

%Given $\a(r)\in\Lambda$ of finite order, a proximate order $\rho(r)$ is said to be a \textit{proximate order of $\a(r)$} if
%$$
%0<\sigma^*[\a]:=\limsup_{r\to\infty}\frac{\a(r)}{r^{\rho(r)}}<\infty.
%$$
%In this case, $\lim_{r\to\infty}\rho(r)=\rho[\a]$, and $\sigma^*[\a]$ does not change if the proximate order is replaced by an equivalent one.
%\end{defi}

\begin{defi}
Let $\ro(r)$ be a proximate order and $f$ be an entire function. The \textit{type of $f$ associated with $\ro(r)$} is
\begin{equation*}
 \sigma_f(\ro(r))=\sigma_f:=\limsup_{r \to \infty} \frac{\log \max_{|z| = r}|f(z)|}{r^{\ro(r)}}.
\end{equation*}
\parn
We say $\ro(r)$ is a \textit{proximate order of $f$} if $ 0<\sigma_f<\infty$.
\end{defi}

\begin{rema}\label{remacotasproximorder}
If  $\ro(r)\to\ro>0$ is a proximate order of $f$, then $f$ is of exponential order $\ro$ and there exists $K>0$ such that for every $z\in\C$ one has
$$|f(z)|\le \exp(K|z|^{\ro(|z|)}).$$
Moreover, and according to Remark~\ref{remaordenaproxequiv}, the type of $f$ does not change if we substitute a proximate order of $f$ by an equivalent one.
\end{rema}

%\begin{prop}(\cite[Thm.\ 2.2.2]{GoldbergOstrowskii})
%Let $\rho(r)$ be a proximate order as in (\ref{defiproxorde}). Then the function $L(r):=r^{\rho(r)-r}$ is \textit{slowly varying}, i.e. for every $t>0$ one has
%$$
%\lim_{r\to\infty}\frac{L(tr)}{L(r)}=1,
%$$
%the convergence being uniform whenever $t$ runs over compact intervals. Moreover, the function $V(r):=r^{\rho(r)}$ is such that for every $t>0$ one has
%$$
%\lim_{r\to\infty}\frac{V(tr)}{V(r)}=t^{\rho},
%$$
%the convergence being uniform whenever $t$ runs over compact intervals.
%\end{prop}
%
%The usefulness of this concept lies on the following result.

%\begin{prop}\label{propexisordeapro}(\cite[Thm.\ 2.2.1]{GoldbergOstrowskii})
%Every continuous function $\a(r)\in\Lambda$ admits a proximate order.
%\end{prop}

The following result of L. S. Maergoiz \cite{Maergoiz} will be the key for our construction.
For an arbitrary sector bisected by the positive real axis, it provides holomorphic functions whose restriction to $(0,\infty)$ is real and has a growth at infinity specified by a prescribed proximate order.

\begin{theo}[\cite{Maergoiz}, Thm.\ 2.4]\label{propanalproxorde}
Let $\ro(r)$ be a proximate order with $\ro(r)\to\ro>0$ as $r\to \infty$. For every $\ga>0$ there exists an analytic function $V(z)$ in $S_\ga$ such that:
  \begin{enumerate}[(i)]
   %\item Para todo $W=(r, \theta)\in L(\ga)$ se tiene que
%   \begin{equation}\label{T24E1:2}
%    \lim_{r \to \infty} \frac{V_0(rW)}{V_0(r)}=W^{\ro}, \quad W \in L(\ga),
%    \end{equation}
%  donde $W^{\ro}= s^\ro e^{i \ro \theta}$ y con convergencia uniforme en los compactos de $L(\ga)$.
\item  For every $z \in S_\ga$,
 \begin{equation*}
    \lim_{r \to \infty} \frac{V(zr)}{V(r)}= z^{\ro},
  \end{equation*}
uniformly in the compact sets of $S_\ga$.
\item $\overline{V(z)}=V(\overline{z})$ for every $z \in S_\ga$ (where, for $z=(|z|,\arg(z))$, we put $\overline{z}=(|z|,-\arg(z))$).
\item $V(r)$ is positive in $(0,\infty)$, monotone increasing and $\lim_{r\to 0}V(r)=0$.
\item The function $t\in\R\to V(e^t)$ is strictly convex (i.e. $V$ is strictly convex relative to $\log(r)$).
\item The function $\log(V(r))$ is strictly concave in $(0,\infty)$.
\item  The function $\ro_0(r):=\log( V(r))/\log(r)$, $r>0$, is a proximate order equivalent to $\ro(r)$.
    \end{enumerate}
\end{theo}

\begin{rema}
We denote by $\mathfrak{B}(\ga,\ro(r))$ the class of such functions $V$. Given a strongly regular sequence $\M$ and its associated function $M(r)$ (see~(\ref{equadefiMdet})), suppose the function $d(r)=\log(M(r))/\log(r)$ is a proximate order. The main results in \cite{sanz13} rested on the fact that, for every $V\in\mathfrak{B}(2\omega(\M),d(r))$, the function $G$ given by
$G(z)=\exp(-V(1/z))$ is a flat function in the class $\tilde{\mathcal{A}}_{\M}(S_{\omega(\M)})$.
We will make use of some of its properties in the next result.
\end{rema}

%\begin{lemma}\label{propiedadese}
%The function $e_V$ enjoys the following properties:
%\begin{itemize}
%\item[(i)] $z^{-1}e_V(z)$ is integrable at the origin, it is to say, for any $t_0>0$ and $\tau\in\R$ with $|\tau|<\frac{\pi\omega(\M)}{2}$ the integral $\int_{0}^{t_0}t^{-1}|e_{V}(te^{i\tau})|dt$ is finite.
%\item[(ii)] For every $T\prec S_{\omega(\M)}$ there exist $C,K>0$ such that
%\begin{equation}\label{e147}|e_V(z)|\le Ch_{\bM}\left(\frac{K}{|z|}\right),\qquad z\in T.
%\end{equation}
%\item[(iii)] For every $x\in\R$, $x>0$, the value $e_V(x)$ is positive real.
%\end{itemize}
%\end{lemma}

\begin{theo}\label{teorconstrkernels}
Suppose $\M$ is a strongly regular sequence with $\omega(\M)<2$ and such that the function $d(r)=\log(M(r))/\log(r)$ is a proximate order. Then, for every $V\in\mathfrak{B}(2\omega(\M),d(r))$, the function $e_V$ defined in $S_{\omega(\M)}$ by
$$
e_V(z)=\frac{1}{\omega(\M)}z\exp(-V(z))
$$
is a kernel of $\M$-summability.
\end{theo}

\begin{proof1}
Since $V$ is holomorphic in $S_{2\omega(\M)}$ and real in $(0,\infty)$, the same is true for $e_V$, so that (\textsc{i}) and (\textsc{iv}) in Definition~\ref{defikernelMsumm} hold. Properties (\textsc{ii}) and (\textsc{iii}) in that Definition have been obtained, with a slight modification in the first case, in Lemma 5.3 of~\cite{sanz13}, as a consequence of the following result.
%The following estimates will be extremely important for determining the asymptotic behaviour at infinity of the kernels.

\begin{prop}[\cite{Maergoiz}, Property\ 2.9]\label{propcotaVpartereal}
 Let $\ro>0$, $\ro(r)$ a proximate order with $\ro(r)\to\ro$, $\ga\ge 2/\ro$ and $V\in \mathfrak{B}(\ga, \ro(r))$. Then, for every $\a\in(0,1/\ro)$ there exist constants $b>0$ and $R_0>0$ such that
 \begin{equation*}
  \Re(V(z)) \ge b V(|z|), \quad  z\in S_{\a},\ |z|\ge R_0.
 \end{equation*}
 \end{prop}
Note that, since $d(r)$ is a proximate order and, by (\ref{equaordeMdet}) and~(\ref{equaordequasM}), we have that
$$
\lim_{r\to\infty}d(r)=\ro[M]=\frac{1}{\omega(\M)},
$$
we may apply Proposition~\ref{propcotaVpartereal} with $\ro=1/\omega(\M)$,
$\ro(r)=d(r)$ and $\ga=2\omega(\M)$.

Then, the moment function associated with $e_V$,
$$m_V(\lambda):=\int_{0}^{\infty}t^{\lambda-1}e_{V}(t)dt=
%\int_{0}^{\infty}t^{\lambda}G_V(1/t)dt=
\int_{0}^{\infty}t^{\lambda}e^{-V(t)}dt,$$
is well defined in $\{\Re(\lambda)\ge0\}$, continuous in its domain and holomorphic in $\{\Re(\lambda)>0\}$; clearly, $m_V(x)>0$ for every $x\ge0$. Moreover, we have the following result.
%We will also make use of the following result.
%
%\begin{theo}[\cite{Thilliez2}, Proposition\ 4]\label{teorcaracfuncplanaAMS}
%Let $\M$ be a strongly regular sequence and $S$ a sector. For $f\in\mathcal{O}(S)$, the following are equivalent:
%\begin{itemize}
%\item[(i)] $f\in\tilde{\mathcal{A}}_{\M}(S)$ and $f\sim_{\M}\hat{0}$.
%\item[(ii)] For every bounded proper subsector $T$ of $S$ there exist $c_1,c_2>0$ with
%$$|f(z)|\le c_1h_{\M}(c_2|z|)=c_1e^{-M(1/(c_2|z|))},\qquad z\in T.
%$$
%\end{itemize}
%\end{theo}

%$V\in\mathfrak{B}(2\omega(\M),d(r))$. We define the \textit{kernel associated with $V$} as $e_V:S_{\omega(\M)}\to\C$ given as
%$$
%e_V(z)=\frac{1}{\omega(\M)}z\exp(-V(z)).
%$$

\begin{prop}[\cite{Maergoiz}, Thm.\ 3.3]
The function
\begin{equation*}%\label{equanucleoE}
E_V(z)=\sum_{n=0}^\infty \frac{z^n}{m_V(n)},\qquad z \in \C,
\end{equation*}
is entire and of proximate order $d_0(r)=\log(V(r))/\log(r)$.
\end{prop}

According to Remark~\ref{remacotasproximorder}, from this fact we deduce that there exists a constant $K_1>0$ such that for every $z\in\C$ one has
$$|E_V(z)|\le \exp(K_1V(|z|)).$$
Since $d_0(r)$ is a proximate order
equivalent to $d(r)=\log(M(r))/\log(r)$, by Remark~\ref{remaordenaproxequiv}
%there exist $K_2>0$ such that $V(r)\le K_2M(r)$ for large $r$, and so
we have
\begin{equation*}%\label{equadesigF_Vordenaprox}
|E_V(z)|\le \tilde{C}\exp(\tilde{K}M(|z|))
\end{equation*}
for every $z\in\C$ and suitably large constants $\tilde{C},\tilde{K}>0$, and so condition (\textsc{v}) in Definition~\ref{defikernelMsumm} is satisfied. Finally, we take into account the following.

\begin{prop}[\cite{Maergoiz}, (3.25)]\label{propcotaVsectorizquier}
Let $\ro(r)$ be a proximate order with $\ro>1/2$, $\ga\ge 2/\ro$ and $V\in\mathfrak{B}(\ga,\ro(r))$. Then, for every $\varepsilon>0$ such that $\varepsilon<\pi(1-1/(2\ro))$ we have, uniformly as $|z|\to \infty$, that (in Landau's notation)
\begin{equation*}
  E_V(z)=%  \begin{cases} ~~V'(z)e^{V(z)}+O\left(\frac{1}{|z|}\right) &\text{if } |\arg z| \le \frac{\pi}{2\ro}+\varepsilon, \\
  O\left(\frac{1}{|z|}\right),\qquad   \frac{\pi}{2\ro}+\varepsilon\le |\arg z| \le \pi.
  %\end{cases}
 \end{equation*}
\end{prop}
This information easily implies that also condition (\textsc{vi})  in Definition~\ref{defikernelMsumm} is fulfilled, what concludes the proof.
\end{proof1}

\begin{rema}
In case $\omega(\M)\ge 2$, we consider $s>0$ and $\M^{(1/s)}:=(M_n^{1/s})_{n\in\N_0}$ as in Remark~\ref{remadefinucleos}(iii), in such a way that $\omega(\M^{(1/s)})=\omega(\M)/s<2$.
With obvious notation, we have that $d^{(1/s)}(r)=sd(r^{s})-\log(s)/\log(r)$ for $r$ large enough, and
$$
r(d^{(1/s)})'(r)\log(r)=sr^{s}d'(r^{s})\log(r^{s})+\frac{\log(s)}{\log(r)}
$$
whenever both sides are defined. So, it is clear that $d(r)$ is a proximate order if, and only if, $d^{(1/s)}(r)$ is. Were this the case, by the previous result we would have kernels $\tilde{e}$ for $\M^{(1/s)}-$sum\-mab\-ility, and the function $e(z)=\tilde{e}(z^{1/s})/s$ will be a kernel for $\M-$summability.
\end{rema}

Regarding the question of whether $d(r)$ is a proximate order or not, we have the following characterization and result.

\begin{prop}[\cite{sanz13}, Prop.\ 4.9]\label{propcaracdderordenaprox}
Let $\M$ be a strongly regular sequence, and $d(r)$ its associated function. The following are equivalent:
\begin{itemize}
\item[(i)] $d(r)$ is a proximate order,
\item[(ii)] $\lim_{p\to\infty} m_pd'(m_p^+)\log(m_p)=0$,
\item[(iii)] $\displaystyle\lim_{p\to\infty}\frac{p+1}{M(m_{p})}=\frac{1}{\omega(\M)}=\ro[M]$.
\end{itemize}
\end{prop}

%\begin{prop}
%$d(r)$ is an increasing function and
%\begin{equation*}
%\lim_{r\to\infty}d(r)=\limsup_{n\to\infty}\frac{\log(n)}{\log(m_{n})}=\frac{1}{\omega(\M)}>0,
%\end{equation*}
%so it is eventually positive.
%\end{prop}\parn
%
%$rd'(r)\log(r)\to 0$ as $r\to\infty$?\par
%
%
%Facts:
%(a) $rd'(r)\log(r)=\frac{rM'(r)}{M(r)}-d(r)=\frac{j}{M(r)}-d(r)$ if $r\in[m_{j-1},m_j)$.\parn
%(b) $\frac{rM'(r)}{M(r)}$ decreases in every $[m_{j-1},m_j)$ and it is positive.
%
%\begin{prop}
%$d(r)$ is a proximate order if, and only if,  $$\lim_{j\to\infty}\frac{j}{M(m_{j-1})}=\lim_{j\to\infty}\frac{j}{\log(m_{j-1}^j/M_j)}
%=\frac{1}{\omega(\M)}.$$
%\end{prop}

%Next we obtain some easy condition that ensures that $d(r)$ is a proximate  order.

\begin{coro}[\cite{sanz13}, Corollary\ 4.10]
If
\begin{equation}\label{equacondordenaprox}
\displaystyle\lim_{p\to\infty}p\log\big(\frac{m_{p+1}}{m_{p}}\big) \textrm{ exists (finite or not),}
\end{equation}
then its value is \emph{a fortiori} $\omega(\M)$, $d(r)$ is a proximate order and, moreover, $$\omega(\M)=\lim_{p\to\infty}\frac{\log(m_p)}{\log(p)}\qquad\textrm{ (instead of $\displaystyle\liminf_{p\to\infty}$, see~(\ref{equaordequasM}))}.
$$
\end{coro}

\begin{rema}\label{remacomentdderordenaprox}
\begin{itemize}
\item[(i)] The previous condition~(\ref{equacondordenaprox}) holds for every sequence $\M_{\a,\beta}$, so that in any of these cases $d(r)$ is a proximate order and it is possible to construct kernels. Indeed, we have not been able yet to provide an example of a strongly regular sequence for which $d(r)$ is not a proximate order, i.e., for which condition (iii) in Proposition~\ref{propcaracdderordenaprox} does not hold.
%\item[(ii)] In the Gevrey case, $\M_{1/k}=(p!^{1/k})_{\in\N_0}$, let us put $M_{1/k}(r)$, $d_{1/k}(r)$, and so on, to denote the corresponding associated functions. Then, one can check (see, for example, \cite{GelfandShilov}) that for large $r$ we have $c_2r^k\le M_{1/k}(r)\le c_1r^k$ for suitable constants $c_1,c_2>0$, so that $\log(c_2)\le(d_{1/k}(r)-k)\log(r)\le\log(c_1)$ eventually. This shows one can work with the constant proximate order $\ro(r)\equiv k$, and any $V\in\mathfrak{B}(2/k,\ro(r))$ will provide us (due to Theorem~\ref{teorconstrfuncplana}, and since $V(r)$ will be bounded above and below by $r^k$ times some suitable constants) with a flat function in the class $\tilde{\mathcal{A}}_{1/k}(S_{1/k})$. It is easy to see that $V(z)=z^{k}$ belongs to $\mathfrak{B}(2/k,\ro(r))$, and we obtain in this way the classical flat function in this situation, namely $G(z)=\exp(-z^{-k})$.
\item[(ii)] If $\M$ is such that $d(r)$ is not a proximate order, but there exists a proximate order $\ro(r)$ and constants $A,B>0$ such that eventually $A\le (d(r)-\ro(r))\log(r)\le B$, then one may also construct kernels for $\M-$summability.
\item[(iii)] The method described in this section provides kernels, but not all. For example, for $k>0$ the function $e(z)=kz^{k}e^{-z^k}$ gives rise to the standard Laplace and Borel (with Mittag-Leffler kernel) transforms of order $k$, and it is a kernel for $\M_{1/k}-$summability. However, it does not arise from the previous construction, as it would correspond to the function $V(r)=r^k-(k-1)\log(r)$ which does not have the required properties.
    %If we put $M_{1/k}(r)$ and $d_{1/k}(r)$ for the associated functions to $\M_{1/k}$, there exist $c_1,c_2>0$ such that for large $r$ we have $c_2r^k\le M_{1/k}(r)\le c_1r^k$, so that $\log(c_2)\le(d_{1/k}(r)-k)\log(r)\le\log(c_1)$ eventually. So, $V(z):=z^{k}\in\mathfrak{B}(2/k,d_{1/k}(r))$, and the kernel $\tilde{e}(z)=kze^{-z^k}$ does appear from the previous construction.
%\item[(iv)] If $e$ is a kernel for $\M$-summability arising from this method and $\a>0$, then one may check that $z^{\a}e(z)$ is also a kernel for $\M$-summability .
\end{itemize}
\end{rema}

\section{Application to some moment-PDE}\label{sectmomentPDE}

%Let us assume that $u$ is a solution of
%$$(\partial_{m_1,t}-\lambda(\partial_{m_2,z} ))u = 0,\qquad u(0,z) =\varphi(z)\in\mathcal{O}(D),$$
%and $\omega(m_1)= q\omega(m_2)$. Suppose the corresponding kernels are related by
%$$
%e_{m_{1}}(z)=\frac{1}{q}e_{m_{2}}(z^{1/q}),\qquad z\in S_{\omega(m_1)}.
%$$
%Then, for every strongly regular sequence $\M=(M_n)_{n\in\N_0}$ (for which kernels of summability exist) and $d\in\R$, the following are equivalent:
%\begin{itemize}
%\item $\varphi$ admits analytic continuation in (a small sector around) every direction of the form
%     $(d+\arg(\lambda)+2p\pi)/q$ for every $p=0,\ldots,q-1$, and it is of exponential $\M$-growth there.
%\item $u$ admits analytic continuation to $\tilde{S}\times D$, where $\tilde S$ is the union of a sector $S$ bisected by $d$ and a disk around 0, and $u$ is uniformly (in $z$) of exponential $\tilde{\M}$-growth in $\tilde S$, where
%    $\tilde{\M}=(M_{qn})_{n\in\N_0}$.
%\end{itemize}

Following the idea of W. Balser and M. Yoshino~\cite{BalserYoshino}, given a sequence of moments $\mathfrak{m}:=(m(p))_{p\in\N_0}$ let us consider the operator $\partial_{\mathfrak{m},z}$, from $\C[[z]]$ into itself, given by
\begin{equation*}%\label{e2a}
\partial_{\mathfrak{m},z}\left(\sum_{p\ge0}\frac{f_{p}}{m(p)}z^{p}\right)=\sum_{p\ge0}\frac{f_{p+1}}{m(p)}z^{p}.
\end{equation*}
S. Michalik~\cite{Michalik3} has studied the initial value problem for linear moment-partial differential equations of the form \begin{equation}\label{e7a}
P(\partial_{\mathfrak{m}_{1},t},\partial_{\mathfrak{m}_{2},z})u(t,z)=0,
\end{equation}
with given initial conditions
\begin{equation}\label{e11a}
\partial_{\mathfrak{m}_{1},t}^{j}u(0,z)=\varphi_{j}(z)\in\mathcal{O}(D),\quad j=0,\ldots,n-1,
\end{equation}
for some $n\in\N$, and some neighborhood of the origin $D$, say $D(0,r)$ for some $r>0$. Here, $P(\lambda,\xi)\in\C[[\lambda,\xi]]$ is a polynomial of degree $n$ in the variable $\lambda$, and $\mathfrak{m}_{1}=(m_1(p))_{p\in\N_0}$ and $\mathfrak{m}_{2}=(m_2(p))_{p\in\N_0}$ are given moment sequences corresponding to kernels $e_1$ and $e_2$ of orders $k_1>0$ and $k_2>0$, respectively, as defined by W. Balser in~\cite{balserutx}. In this last section we aim at stating analogous results to those in~\cite{Michalik3}, now in the case when these kernels are associated with general strongly regular sequences (which might not be equivalent to Gevrey ones). So, our setting is as described in Remark \ref{remamomentsarestronglyregular}(ii).
Although the class of linear moment-partial differential equations under study has been enlarged, the main ideas do not greatly differ from the ones in~\cite{Michalik3}, so we will omit some proofs requiring only minor modifications with respect to the ones provided in that work.

%Through the whole section, we assume $\bM$ is a strongly regular sequence with $w(\bM)<2$  and such that $d(r)$  is a proximate order. From Theorem~\ref{}
%
%%%%%%%%%%%%
%{\color{red}{THEOREM 4.7.  HACER LA REFERENCIA}}
%%%%%%%%%%%%
%
%, one can construct a function $e$ to be a kernel for $\bM$-summability. Let $m_{e}$ be the moment function associated with $e$, and $m:=m_{e}=(m_e(p))_{p\in\N_0}$ be the sequence of moments associated with $e$.
%
%We also follow the previous construction for two given strongly regular sequences $\bM^{(1)}$ and $\bM^{(2)}$ with $w(\bM^{(j)})<2$ with proximate orders $d_{j}(r)$ for $j=1,2$, leading to the corresponding kernel functions $e_j$ and the corresponding sequences of moments $m_{j}$, for $j=1,2$.
%
%In this section, we deal which the Cauchy problem (\ref{e7a}),(\ref{e11a}) under these more general settings, where the definition of formal moment derivation stated in (\ref{e2a}) is now extended to sequences of moments derived from strongly regular sequences.

The approach in~\cite{Michalik3} is based on the reduction of the initial problem (\ref{e7a}),(\ref{e11a}) into a finite number of problems which are easier to handle. For this purpose, we put
\begin{equation}\label{e30a}
P(\lambda,\xi)=P_0(\xi)(\lambda-\lambda_{1}(\xi))^{n_{1}}\cdots(\lambda-\lambda_{\ell}(\xi))^{n_{\ell}},
\end{equation}
where $n_{1},\ldots,n_{\ell}\in\N$ with $n_1+\cdots+n_{\ell}=n$. For every $j=1,\ldots,\ell$, the function $\lambda_{j}(\xi)$ is an algebraic function, holomorphic for $|\xi|>R_0$, for some $R_0>0$, and with polynomial growth at infinity. The existence is proven of a normalised formal solution $\hat{u}$ to the main problem (\ref{e7a}),(\ref{e11a}), chosen so as to satisfy also the  equation
\begin{equation}\label{equanormalised}
(\partial_{\mathfrak{m}_1,t}-\lambda_1(\partial_{\mathfrak{m}_2,z}))^{n_1}\cdots
(\partial_{\mathfrak{m}_1,t}-\lambda_{\ell}(\partial_{\mathfrak{m}_2,z}))^{n_\ell}\hat{u}=0
\end{equation}
(the meaning of $\lambda_j(\partial_{\mathfrak{m}_2,z})$ to be specified).
Indeed, Theorem 1 in~\cite{Michalik3} states that one can recover $\hat{u}$ as
\begin{equation}\label{equasolunormalizada}
\hat{u}=\sum_{\alpha=1}^{\ell}\sum_{\beta=1}^{n_{\alpha}}\hat{u}_{\alpha\beta},
\end{equation}
$\hat{u}_{\alpha\beta}$ being the formal solution of
\begin{equation}\label{e87a}
\left\{
\begin{array}{l}
(\partial_{\mathfrak{m}_1,t}-\lambda_{\alpha}(\partial_{\mathfrak{m}_2,z}))^{\beta}\hat{u}_{\alpha\beta}=0\\
\partial_{\mathfrak{m}_1,t}^{j}\hat{u}_{\alpha\beta}(0,z)=0,\quad j=0,\ldots,\beta-2\\
\partial_{\mathfrak{m}_1,t}^{\beta-1}\hat{u}_{\alpha\beta}(0,z)=
\lambda_{\alpha}^{\beta-1}(\partial_{\mathfrak{m}_2,z})\phi_{\alpha\beta}(z),
\end{array}
\right.
\end{equation}
where $\phi_{\alpha\beta}(z):=\sum_{j=0}^{n-1}d_{\alpha\beta j}(\partial_{\mathfrak{m}_2,z})\phi_{j}(z)\in\mathcal{O}(D(0,r))$, and $d_{\alpha\beta j}(\xi)$ are holomorphic functions of polynomial growth at infinity for every $\alpha$ and $\beta$. One may easily check that the formal solution of (\ref{e87a}) is given by
\begin{equation}\label{e97a}
\hat{u}_{\alpha\beta}(t,z)=
\sum_{j=\beta-1}^{\infty}\left(\begin{array}{c}j\\\beta-1\end{array}\right)
\frac{\lambda_{\alpha}^{j}(\partial_{\mathfrak{m}_{2},z})\phi_{\alpha\beta}(z)}{m_{1}(j)}t^{j}.
\end{equation}
We do not enter into details about this point, for the proof of this result is entirely analogous in our situation. We will focus our attention on the convergence of the formal solution, and also on the growth rate of its coefficients when it has null radius of convergence, but firstly we recall the meaning of the pseudodifferential operators $\lambda(\partial_{\mathfrak{m}_e,z})$ (like the ones appearing in~(\ref{equanormalised}), (\ref{e87a}) and~(\ref{e97a})), where $\lambda(\xi)$ is an element in the set $\{\lambda_{j}: j=1,\ldots,\ell\}$ and $\mathfrak{m}_e=(m_e(p))_{p\in\N_0}$ is the strongly regular sequence of moments of a kernel $e$ with moment function $m_e(\lambda)$.
Given $r>0$, one can check (see Proposition~3 in~\cite{Michalik3}) that the differential operator $\partial_{\mathfrak{m}_e,z}$ is well-defined for any $\phi\in\mathcal{O}(D(0,r))$, and for $0<\varepsilon<r$ and every $z\in D(0,\varepsilon)$ one has
\begin{equation*}%\label{e32a}
\partial_{\mathfrak{m}_e,z}^{n}\phi(z)=\frac{1}{2\pi i}\oint_{|w|=\varepsilon}\phi(w)\int_{0}^{\infty(\theta)}\xi^nE(z\xi)\frac{e(w\xi)}{w\xi}d\xi dw,
\end{equation*}
for every $n\in\N_0$, where $\theta\in(-\arg(w)-\frac{w(\mathfrak{m}_e)\pi}{2},-\arg(w)+\frac{w(\mathfrak{m}_e)\pi}{2})$ and $E$ is the second kernel in Definition~\ref{defikernelMsumm}.
%It is worth remarking that the kernel functions $E_m$ and $e_m$ satisfy analogous properties as the corresponding ones used in~\cite{Michalik3}, as stated in Section~\ref{}
%
%%%%%%%%%%%%%%%%%%%%%%%%%%%%
%{\color{red}{SECTION 3. INCLUIR EN LAS REFERENCIAS}}
%%%%%%%%%%%%%%%%%%%%%%%%%%%%

%However, they are related to a sequence of moments which might not be equivalent to a Gevrey one, although calculations do not differ from those.

%Let $\lambda(\xi)$ be an element in the set $\{\lambda_{j}: j=1,\ldots,\ell\}$. Namely, $\lambda$ is well-defined for $|\lambda|$ large enough, and has polynomial growth at infinity. The writing of $\partial_{m,z}$ stated in (\ref{e32a})
This expression inspires the definition of the pseudodifferential operator $\lambda(\partial_{\mathfrak{m}_e,z})$ as
$$\lambda(\partial_{\mathfrak{m}_e,z})\phi(z):=\frac{1}{2\pi i}\oint_{|w|=\varepsilon}\phi(w)\int_{\xi_0}^{\infty (\theta)}\lambda(\xi)E(\xi z)\frac{e(\xi w)}{\xi w}d\xi dw,$$
for every $\phi\in \mathcal{O}(D(0,r))$, where $\xi_0=R_0e^{i\theta}$ with suitably large $R_0>0$ and $\theta$ as before (see Definition~8 in~\cite{Michalik3}).

Next we study the growth rate of the formal solution of (\ref{e87a}), given in (\ref{e97a}).
To this end, we need the following definition and lemmas.

\begin{defi}\label{defipoleorder}
Let $U\subseteq \C$ be a neighborhood of $\infty$, and $\Psi\in\mathcal{O}(U)$. The \textit{pole order} $q\in\mathbb{Q}$ and the \textit{leading term} $\psi\in\C\setminus\{0\}$ associated with $\Psi$ are the elements satisfying
$$\lim_{z\to\infty}\frac{\Psi(z)}{z^{q}}=\psi,$$
if they exist.
\end{defi}

%As it was mentioned before, the previous definition aims to generalize the structure motivated in (\ref{e32a}). Moreover, the elements in the image of such an operator can be estimated in terms of the sequence $\bM$.

%The next technical result will be useful in forthcoming estimates.
%relates the values of a moment sequence evaluated at any positive real number with the elements in the related strongly regular sequence.

%\begin{lemm}\label{e78a}
%Let $m_1:=(m_{1}(p))_{p\ge0}$, $m_{2}:=(m_{2}(p))_{p\ge0}$ be as above. Then, the sequence $m_{3}=(m_{3}(p))_{p\ge0}$ defined by $m_{3}(p):=m_{1}(p)m_{2}(p)$ for $p\in\N_0$ satisfies that
%$$h_{m_{3}}(t)\ge h_{m_{1}}(t^{1/2})h_{m_{2}}(t^{1/2}),$$
%for every $t\ge0$. Moreover, $m_3$ turns out to be the moment sequence associated with the kernel function $e_{1}\star e_2$.
%\end{lemm}
%\begin{proof}
%The fact that $m_{3}$ is a sequence of moments can be deduced from the results in \cite{balserutx}. Indeed, if $e_{j}$ is the kernel function associated with $m_{j}$ for $j=1,2$, then $e_{1}\star e_{2}$ turns out to be a kernel function associated with $m_{3}$. Here, $\star$ stands for the convolution product.
%
%%%%%%%%%%%%%%%%%%%
%{\color{red}{COMO EN BALSER, THEOREM 31. ESTA COMPROBADO. MERECE LA PENA PONERLO? O PARA MULTISUMABILIDAD?}}
%%%%%%%%%%%%%%%%%%
%
%
%The inequality stated in the enunciate is straight from the definition of $h_{m_3}$.
%
%\end{proof}

%In the next lemma, we preserve the space where the functions constructed rest after the action of the operator $\lambda(\partial_{m,z})$.

\begin{lemm}\label{lema57a}
Let $e$, $m_e$ and $\mathfrak{m}_e$ be as before, $\lambda(\xi)$ have pole order $q$ and leading term $\lambda_0$, and let $\phi\in\mathcal{O}(D(0,r))$. There exist $r_0,A,B>0$ such that
% for every moment pseudodifferential operator $\lambda(\partial_{m,z})$, one has
$$\sup_{|z|<r_0}\left|\lambda(\partial_{\mathfrak{m}_e,z})\phi(z)\right|\le |\lambda_0|A B^{q}m_e(q).$$
%where $\lambda$ and $q$ are the leading term and the pole order of $\lambda(\xi)$, respectively.
\end{lemm}
\begin{proof}
One may choose $R_0>0$ such that $|\lambda(\xi)|\le 2|\lambda_0||\xi|^q$ for every $\xi$ with $|\xi|\ge R_0$. Let $w\in\C$ with $0<|w|=\varepsilon<r$, and put $\theta=-\arg(w)$ and $\xi_0=R_0e^{i\theta}$. One has
\begin{equation}\label{e64a}
\left|\int_{\xi_0}^{\infty(\theta)}\lambda(\xi)E(\xi z)\frac{e(\xi w)}{\xi w}d\xi\right|
\le 2|\lambda_0|\int_{R_0}^{\infty}s^q|E(se^{i\theta}z)|\frac{|e(se^{i\theta}w)|}{s\varepsilon}ds.
\end{equation}
The properties of the kernel functions $e$ and $E$ stated in Definition~\ref{defikernelMsumm}, rephrased according to Remark~\ref{remamomentsarestronglyregular}(ii), allow us to write
\begin{equation*}%\label{e68a}
|E(se^{i\theta}z)e(se^{i\theta}w)|\le\frac{c_1}{h_{\mathfrak{m}_e}\left(\frac{c_2}{s|z|}\right)}
h_{\mathfrak{m}_e}\left(\frac{c_3}{s\varepsilon}\right)
\end{equation*}
for some $c_1,c_2,c_3>0$ and for every $s\in[R_0,\infty)$.
%, and $\arg(se^{i\theta}w)=\theta+\arg(w)\in\left(-\frac{w(\bM)\pi}{2},\frac{w(\bM)\pi}{2}\right)$. Ona can assume that, for every $\theta$, $w$ is taken so that $\theta+\arg(w)\in\left(-\frac{w(\bM)\pi}{2}+\varepsilon_1,\frac{w(\bM)\pi}{2}-\varepsilon_1\right)$ for some $\varepsilon_1>0$.
>From (\ref{e120}) one has
\begin{equation}
\label{e74a}
\frac{c_1}{h_{\mathfrak{m}_e}\big(\frac{c_2}{s|z|}\big)}h_{\mathfrak{m}_e}\big(\frac{c_3}{s\varepsilon}\big)\le \frac{c_1h^2_{\mathfrak{m}_e}\big(\frac{\rho(2)c_3}{s\varepsilon}\big)}{h_{\mathfrak{m}_e}\big(\frac{c_2}{s|z|}\big)}.
\end{equation}
Let $r_0>0$ be such that $r_0\le c_2\varepsilon/(\rho(2)c_3)$, so that $\rho(2)c_3/(s\varepsilon)\le c_2/(s|z|)$ for every $z\in D(0,r_0)$. For such $z$, the expression in the right-hand side of (\ref{e74a}) is upper bounded by $c_1h_{\mathfrak{m}_e}\big(\rho(2)c_3/(s\varepsilon)\big)$, and one obtains that the last expression in (\ref{e64a}) can be upper bounded by
$$2c_1|\lambda_0|\int_{R_0}^{\infty}\frac{s^{q-1}}{\varepsilon}h_{\mathfrak{m}_e}\big(\frac{\rho(2)c_3}{s\varepsilon}\big)ds.$$
In turn, by the very definition of $h_{\mathfrak{m}_e}$, the previous quantity is less than
$$2c_1|\lambda_0|(\rho(2)c_3)^{\left\lfloor q\right\rfloor+2}\frac{1}{\varepsilon^{\left\lfloor q\right\rfloor+3}}m_e(\left\lfloor q\right\rfloor+2)\int_{R_0}^{\infty}\frac{1}{s^{3+\left\lfloor q\right\rfloor-q}}ds.$$
The last integral is easily seen to be bounded above by some constant independent of $q$. Moreover, the moderate growth property of $\mathfrak{m}_e$ leads to an estimate of the form
$$
|\lambda_0|A_0B_0^{\lfloor q\rfloor}m_e(\lfloor q\rfloor).
$$
Finally, we observe that the function $x\in[0,\infty)\to m_e(x)$ is continuous, strictly convex (since $m_e''(x)>0$ for every $x>0$) and $\lim_{x\to\infty}m_e(x)=\infty$, so it reaches its absolute minimum $m_e(x_0)>0$ at a point $x_0\ge 0$, and it is decreasing in $[0,x_0)$ (if $x_0>0$) and increasing in $(x_0,\infty)$. So, we deduce that whenever $x_0\le x\le y$ we have $m_e(x)\le m_e(y)$, while if $0\le x<x_0$ and $x\le y$, then $m_e(x)/m_e(y)\le m_e(0)/m_e(x_0)$. In conclusion, there exists a constant $A_1>0$ such that $m_e(x)\le A_1m_e(y)$ whenever $0\le x\le y$, and in particular, $m_e(\lfloor x\rfloor)\le A_1m_e(x)$ for every $x>0$, what leads to the final estimate.
\end{proof}

For $j\in\N$, the function $\lambda^j(\xi)$ has pole order $jq$ and leading term $\lambda_0^j$. So, an argument similar to the previous one provides the proof for the following result.

\begin{coro}\label{coro84a}
Let $j\in\N$. Under the assumptions of Lemma~\ref{lema57a}, one has
$$\sup_{|z|<r_0}|\lambda^{j}(\partial_{\mathfrak{m}_e,z})\phi(z)|\le |\lambda_0|^{j}AB^{jq}m_e(qj)
$$
for some $r_0,A,B>0$.
\end{coro}

\begin{rema}
The previous estimates, according to Remark~\ref{remanotassumabM}(ii), could also be expressed as $|\lambda_0|^{j}AB_1^{jq}(m_e(j))^q$ for suitable $B_1>0$.
\end{rema}

As indicated before (see~(\ref{e87a}) and~(\ref{e97a})), a problem of the form
\begin{equation}\label{e116a}
\left\{
\begin{array}{l}
(\partial_{\mathfrak{m}_1,t}-\lambda(\partial_{\mathfrak{m}_2,z}))^{\beta}\hat{u}=0\\
\partial_{\mathfrak{m}_1,t}^{j}\hat{u}(0,z)=0,\quad j=0,\ldots,\beta-2\\
\partial_{\mathfrak{m}_1,t}^{\beta-1}\hat{u}(0,z)=\lambda^{\beta-1}(\partial_{\mathfrak{m}_2,z})\phi(z),
\end{array}
\right.
\end{equation}
where $\beta\in\N$ and $\phi\in\mathcal{O}(D(0,r))$, has
\begin{equation}\label{e97b}
\hat{u}(t,z)=\sum_{j=\beta-1}^{\infty}\left(\begin{array}{c}j\\\beta-1\end{array}\right)
\frac{\lambda^{j}(\partial_{\mathfrak{m}_{2},z})\phi(z)}{m_{1}(j)}t^{j}
=:\sum_{j=0}^{\infty}u_j(z)t^{j}
\end{equation}
as its formal solution, and Corollary~\ref{coro84a} allows us to claim that
$$\sup_{|z|<r_0}|u_{j}(z)|\le CD^{j}\frac{m_{2}(qj)}{m_1(j)},$$
for some $r_0,C,D>0$ and for every $j\in\N_0$. Hence, convergence or divergence of $\hat{u}$ in some neighborhood of the origin is a consequence of the growth rate of the sequence $(\frac{m_{2}(qj)}{m_1(j)})_{j\ge0}$. More precisely, one has

\begin{coro}\label{coro135a}
If
$$
\overline{\lim}_{j\to\infty}\left(\frac{m_{2}(qj)}{m_1(j)}\right)^{1/j}<\infty,
$$
then $\hat{u}$ in~(\ref{e97b}) defines a holomorphic function $u(t,z)$ on $D_{1}\times D(0,r_0)$ for some neighborhood of the origin $D_1\subseteq\C$, and $u$ solves~(\ref{e116a}).
\end{coro}

We now turn our attention to the determination of sufficient conditions for $u(t,z)$ to admit analytic continuation in an unbounded sector with respect to the variable $t$ and with adequate growth. We first need some notation, starting with $\mathcal{O}^{\mathfrak{m}_e}(S)$ (see Definition~\ref{defiMgrowth}), where $S$ is an unbounded sector in $\mathcal{R}$ and $\mathfrak{m}_e$ is a strongly regular sequence of moments for a kernel $e$.

\begin{defi}
We write $f\in\mathcal{O}^{\mathfrak{m}_e}(\hat{S})$ if $f\in\mathcal{O}^{\mathfrak{m}_e}(S)\cap\mathcal{O}(S\cup D)$ for some disc $D=D(0,r)$.

Let $D=D(0,r)$. We say $f(t,z)$, holomorphic in $S\times D$, belongs to $\mathcal{O}^{\mathfrak{m}_e}(S\times D)$ if for every $T\prec S$ and $r_1\in(0,r)$ there exist $c,k>0$ such that
$$\sup_{z\in D(0,r_1)}|f(t,z)|\le \frac{c}{h_{\mathfrak{m}_e}\left(k/|t|\right)},\quad t\in T.$$

Analogously, we write $f\in\mathcal{O}^{\mathfrak{m}_e}(\hat{S}\times D)$ if $f\in\mathcal{O}((S\cup D_1)\times D)\cap \mathcal{O}^{\mathfrak{m}_e}(S\times D)$ for some disc  $D_1$ around the origin. We also write $\mathcal{O}^{\mathfrak{m}_e}(\hat{S}(d))$ and $\mathcal{O}^{\mathfrak{m}_e}(\hat{S}(d)\times D)$ whenever the sector $S$ is of the form $S(d,\varepsilon)$ for some inessential $\varepsilon>0$.
\end{defi}

>From Proposition~\ref{propKomatsu} we deduce the following result.

\begin{coro}\label{coro136a}
If there exists a strongly regular sequence of moments $\mathfrak{m}_e=(m_e(j))_{j\in\N_0}$ and  $C,D>0$ such that
$$m_2(qj)m_e(j)\le CD^{j}m_1(j)$$
for every $j\ge 0$, then $\hat{u}$ in~(\ref{e97b}) defines a function $u\in\mathcal{O}(\C\times D(0,r_0))$, and one has
$$\sup_{|z|<r_0}|u(t,z)|\le \frac{c}{h_{\mathfrak{m}_e}(k/|t|)},$$
for some $c,k>0$ and for every $t\in\C$.
\end{coro}

\begin{rema}
In the particular case that $\mathfrak{m}_{1}=\bM_{1/k_{1}}$ and $\mathfrak{m}_{2}=\bM_{1/k_{2}}$ for some $k_1,k_2>0$ with $1/k_1>q/k_2$, we would have that $\hat{u}\in\mathcal{O}(\C\times D(0,r_0))$, with exponential growth in the variable $t$ of order $(\frac{1}{k_1}-q\frac{1}{k_2})^{-1}$, namely
$$\sup_{|z|<r_0}|u(t,z)|\le Ce^{D|t|^{\frac{k_1k_2}{k_2-qk_1}}},\quad t\in\C,$$
for some $C,D>0$, as stated in Proposition 5 of~\cite{Michalik3}.
\end{rema}

In order to go further in our study, by an argument entirely analogous to that in Lemma 4 in~\cite{Michalik3} one can prove that, under the assumptions of Corollary~\ref{coro135a}, the actual solution of (\ref{e116a}) can be written in a neighborhood of $(0,0)$ in the form
\begin{equation}\label{eB}
u(t,z)=\frac{t^{\beta-1}}{(\beta-1)!}\partial_{t}^{\beta-1}\frac{1}{2\pi i}\oint_{|w|=\varepsilon}\phi(w)\int_{\xi_0}^{\infty(\theta)}E_{1}(t\lambda(\xi))E_{2}(\xi z)\frac{e_{2}(\xi w)}{\xi w}d\xi dw,
\end{equation}
with $\theta\in(-\arg(w)-\frac{\pi}{2}w(\mathfrak{m}_2), -\arg(w)+\frac{\pi}{2}w(\mathfrak{m}_2))$,
and where $E_1$ and $E_2$ are the kernels corresponding to $e_1$ and $e_2$, respectively.

We are ready to relate the properties of analytic continuation and growth of the initial data with those of the solution. In these last results we assume the kernels $e_1,E_1$ have been constructed following the procedure in Section~\ref{sectkernelsfromproxorder}.

\begin{lemm}\label{lemarelaccrecexpondatossoluc}
Let $q=\mu/\nu\in\mathbb{Q}$, with $\gcd(\mu,\nu)=1$ and $\beta\ge1$. We assume the moment functions $m_1(\lambda)$ and $m_2(\lambda)$ are such that
\begin{equation}\label{eA}
m_2(qj)\le C_0A_0^jm_1(j), \quad j\in\N_0,
\end{equation}
and
\begin{equation}\label{eAbis}
m_1(j/q)\le C_1A_1^jm_2(j), \quad j\in\N_0,
\end{equation}
for suitable $C_0,C_1,A_0,A_1>0$. Let $u(t,z)$ be a solution of
\begin{equation}\label{e193a}
\left\{
\begin{array}{l}
(\partial_{\mathfrak{m}_1,t}-\lambda(\partial_{\mathfrak{m}_2,z}))^{\beta}u=0\\
\partial_{\mathfrak{m}_1,t}^{j}\hat{u}(0,z)=\phi_{j}(z)\in\mathcal{O}(D(0,r)),\quad j=0,\ldots,\beta-1,
\end{array}
\right.
\end{equation}
for some $r>0$. If there exists a strongly regular sequence of moments $\mathfrak{m}=(m(j))_{j\in\N_0}$ such that:
\begin{itemize}
\item[(i)] there exist $C,A>0$ with
\begin{equation}\label{equammenormuno}
m(j)\le CA^jm_1(j), \quad j\in\N_0,
\end{equation}
\item[(ii)] $\phi_{j}\in\mathcal{O}^{\mathfrak{m}^{(1/q)}}(\hat{S}((d+\arg(\lambda))/q+2k\pi/\mu))$ for every $k=0,\ldots,\mu-1$ and $j=0,\ldots,\beta-1$, and some $d\in\R$,
\end{itemize}
then $u(t,z)\in\mathcal{O}^{\mathfrak{m}}(\hat{S}(d+2n\pi/\nu)\times D(0,r))$ for $n=0,\ldots,\nu-1$.
% Here, the product sequence $\mathfrak{m}_2\cdot\mathfrak{m}'$ is defined as in Proposition~\ref{propprodfuertregu}.
\end{lemm}

\begin{rema}
According to Remark~\ref{remanotassumabM}(ii), $(m_2(qj))_{j\in\N_0}$
(respectively, $(m_1(j/q))_{j\in\N_0}$) is equivalent to $\mathfrak{m}_2^{(q)}$ (resp. to $\mathfrak{m}_1^{(1/q)}$). Together with this fact, the inequalities (\ref{eA}) and (\ref{eAbis}) amount to the equivalence of $\mathfrak{m}_2^{(q)}$ and $\mathfrak{m}_1$, and so, we deduce by (\ref{equaomegamomentosescalados}) that
\begin{equation}\label{equaomegaigualqomega}
q\o(\mathfrak{m}_2)=\o(\mathfrak{m}_1).
\end{equation}
\end{rema}

\begin{proof}
%Observe that, under the assumption made on the moment sequences, the principle of superposition of solutions of linear equations and Corollary~\ref{coro135a} guarantee the existence of a holomorphic solution $u(t,z)$ of (\ref{e193a}) defined on some neighborhood of the origin in $\C^2$.
With the help of Lemma 3 in~\cite{Michalik3}, which may be reproduced in our setting without modification, one can show that the general situation may always be taken into the case $\omega(\mathfrak{m}_1)<2$, which will be the only one we consider.
The principle of superposition of solutions allows us to reduce the study of (\ref{e193a}) to that of some problems of the form (\ref{e116a}), where $\lambda^{\beta-1}(\partial_{\mathfrak{m}_2,z})\phi$ in (\ref{e116a}) turns out to be a function belonging to $\mathcal{O}^{\mathfrak{m}^{(1/q)}}(\hat{S}((d+\arg(\lambda))/q+2k\pi/\mu))$ for every $k=0,\ldots,\mu-1$ and $j=0,\ldots,\beta-1$. Moreover, Corollary~\ref{coro135a} and (\ref{eA}) guarantee the existence of a holomorphic solution $u(t,z)$ of (\ref{e193a}), defined on some neighborhood of the origin in $\C^2$, which can be written in the form (\ref{eB}). Next, we claim that the function
\begin{equation}\label{e212a}
t\mapsto\int_{\xi_0}^{\infty(\theta)}E_{1}(t\lambda(\xi))E_{2}(\xi z)\frac{e_{2}(\xi w)}{\xi w}d\xi,
\end{equation}
which is holomorphic in $\{t\in\C:|t|\le C_2 |w|^{q}\}$ for some $C_2>0$, can be analytically continued to the set
$$\Omega=\{t\in\mathcal{R}:\arg(t)+2k\pi+\arg(\lambda)\neq(\arg(w)+2n\pi)q \textrm{ for every }k,n\in\mathbb{Z}\}.$$
%Indeed, because of Remark~\ref{remaomegamonotona} and~(\ref{equaomegamomentosescalados}), we have that inequalities~(\ref{eA}) imply
%$$
%q\o(\mathfrak{m}_2)=\o((m_2(qj))_{j\in\N_0})\le \o((C_0A_0^jm_1(j))_{j\in\N_0})=\o(\mathfrak{m}_1).
%$$
%Analogously, inequalities~(\ref{eAbis}) imply $\o(\mathfrak{m}_1)/q\le \o(\mathfrak{m}_2)$, and so
Indeed, the equality (\ref{equaomegaigualqomega}) entails that, as long as $t\in\Omega$,
one can replace $\theta$ in (\ref{e212a}) by a direction $\tilde{\theta}$ such that
$$
\arg(t)+2k\pi+\arg(\lambda)+q\tilde{\theta}\in\big(\pi\o(\mathfrak{m}_1)/2, 2\pi-\pi\o(\mathfrak{m}_1)/2\big) \quad\textrm{for some $k\in\Z$}
$$
and
$$
\arg(w)+2n\pi+\tilde{\theta}\in\big(-\pi\o(\mathfrak{m}_2)/2, \pi\o(\mathfrak{m}_2)/2\big) \quad\textrm{for some $n\in\Z$},
$$
what makes the continuation possible by ensuring the adequate asymptotic behavior of the integrand as $\xi\to\infty$, $\arg(\xi)=\tilde{\theta}$.
The rest of the proof, intended to estimate $u$, also follows the arguments in~\cite[Lemma\ 5]{Michalik3}, but estimates will be carefully given in order to highlight the techniques in this general situation. Suppose $z$ is small relative to $w$.
We deform the integration path $|w|=\varepsilon$ in order to write
\begin{equation*}%\label{e216a}
u(t,z)=\frac{t^{\beta-1}}{(\beta-1)!}\partial_{t}^{\beta-1}\left(u_{1}(t,z)+u_{2}(t,z)\right),
\end{equation*}
with
$$u_1(t,z)=\sum_{k=0}^{\mu-1}\frac{1}{2\pi i}\oint_{\gamma_{2k}^{R}}\phi(w)\int_{\xi_0}^{\infty(\theta)}E_{1}(t\lambda(\xi))E_{2}(\xi z)\frac{e_{2}(\xi w)}{\xi w}d\xi dw,$$
and
$$u_2(t,z)=\sum_{k=0}^{\mu-1}\frac{1}{2\pi i}\oint_{\gamma_{2k+1}}\phi(w)\int_{\xi_0}^{\infty(\theta)}E_{1}(t\lambda(\xi))E_{2}(\xi z)\frac{e_{2}(\xi w)}{\xi w}d\xi dw.$$

Here, the path $\gamma_{2k+1}$ is parameterized by
$$
s\in I_{2k+1}:=\left(\frac{d+\arg(\lambda)}{q}+\frac{2k\pi}{\mu}+\frac{\delta}{3}, \frac{d+\arg(\lambda)}{q}+\frac{2(k+1)\pi}{\mu}-\frac{\delta}{3}\right)\mapsto \varepsilon e^{is},
$$
for some small enough $\delta>0$. On the other hand, for large enough $R>0$ the path $\gamma_{2k}^{R}$ is $\gamma_{2k}^{R,-}+\gamma_{2k}^{R,1}-\gamma_{2k}^{R,+}$, where
$$
\gamma_{2k}^{R,\star}(s)= se^{i\left(\frac{d+\arg(\lambda)}{q}+\frac{2k\pi}{\mu}\star\frac{\delta}{3}\right)}=se^{i\theta_{\star}},\qquad \star\in\{-,+\},\quad s\in[\varepsilon,R],
$$
and
$$\gamma_{2k}^{R,1}(s)=Re^{is},\qquad s\in\left(\frac{d+\arg(\lambda)}{q}-\frac{2k\pi}{\mu}-\frac{\delta}{3}, \frac{d+\arg(\lambda)}{q}+\frac{2k\pi}{\mu}+\frac{\delta}{3}\right).
$$

We now give growth estimates for $u_1$ and $u_2$ in order to conclude the result. We first give bounds for $u_2(t,z)$. We take $k\in\{0,\ldots,\mu-1\}$. Let $t$ be as above with $|t|\ge 1$, and consider $\xi$ and $w$ in the trace of the corresponding path defined by the path integrals in the definition of~$u_2$. From the properties of the kernel functions in Definition~\ref{defikernelMsumm}, one has that
$$\left|E_{1}(t\lambda(\xi))E_2(\xi z)\frac{e_{2}(\xi w)}{\xi w}\right|\le C_{11}|E_{1}(t\lambda(\xi))|\frac{h_{\mathfrak{m}_2}\left(\frac{C_{12}}{|\xi|\varepsilon}\right)}{h_{\mathfrak{m}_2}\left(\frac{C_{13}}{|\xi||z|}\right)|\xi|\varepsilon},$$
for some $C_{11},C_{12},C_{13}>0$.

Taking into account (\ref{e120}) we have
\begin{align*}J_{1}&:=\left|  \oint_{\gamma_{2k+1}}\phi(w)\int_{\xi_0}^{\infty(\theta)}E_{1}(t\lambda(\xi))E_{2}(\xi z)\frac{e_{2}(\xi w)}{\xi w}d\xi dw\right|\\
%\label{e1690}
&\le C_{11}\int_{s\in I_{2k+1}}\left|\phi(\varepsilon e^{is})\right|
\int_{|\xi_{0}|}^{\infty}h_{\mathfrak{m}_2}\left(\frac{\rho(2)C_{12}}{|\xi|\varepsilon}\right)|E_{1}(t\lambda(|\xi|e^{i\theta}))| \frac{h_{\mathfrak{m}_2}\left(\frac{\rho(2)C_{12}}{|\xi|\varepsilon}\right)}{h_{\mathfrak{m}_2}\left(\frac{C_{13}}{|\xi||z|}\right)|\xi|\varepsilon}d|\xi|ds.
\end{align*}

We assume $z$ satisfies $|z|\le C_{13}\varepsilon/(\rho(2)C_{12})$. This entails
\begin{equation*}%\label{e1691}
h_{\mathfrak{m}_2}\left(\frac{\rho(2)C_{12}}{|\xi|\varepsilon}\right) \le h_{\mathfrak{m}_2}\left(\frac{C_{13}}{|\xi||z|}\right).
\end{equation*}
By the careful choice of the direction $\theta$ above and because of Proposition~\ref{propcotaVsectorizquier} applied to $E_1$, we deduce there exists $\delta>0$ such that the function $(t,|\xi|)\mapsto |E_1(t\lambda(|\xi|e^{i\theta}))|$ admits a maximum at a point, say $(t_1,|\xi_{1}|)$, as $(t,|\xi|)$ runs over
$(S(d+2n\pi/\nu,\delta)\cap \{t:|t|\ge 1\})\times[|\xi_0|,\infty)$. Then, for every such $t$ and $|\xi|\ge |\xi_0|$, one easily obtains constants $C_{14},C_{15}>0$ such that
$$|E_1(t\lambda(|\xi|e^{i\theta}))|\le |E_1(t\lambda(|\xi_1|e^{i\theta}))|\le\frac{C_{14}}{h_{\mathfrak{m}_1}\left(\frac{C_{15}}{|t|}\right)}.$$
Moreover,
$$\int_{|\xi_{0}|}^{\infty}h_{\mathfrak{m}_2}\left(\frac{\rho(2)C_{12}}{|\xi|\varepsilon}\right)\frac{1}{|\xi|\varepsilon}d|\xi|<\infty,$$
so
$$J_{1}\le \frac{C_{16}}{h_{\mathfrak{m}_1}\left(\frac{C_{15}}{|t|}\right)}\int_{s\in I_{2k+1}}\left|\phi(\varepsilon e^{is})\right|ds.$$
Taking into account that
$$\sup_{|w|=\varepsilon}|\phi(w)|<\infty,$$
one concludes that $u_{2}\in\mathcal{O}^{\mathfrak{m}_{1}}(\hat{S}(d+2n\pi/\nu)\times D(0,r))$ for some $r>0$ and for $n=0,\ldots,\nu-1$.

We now give estimates on $u_{1}(t,z)$.  The inner integral in the definition of each term in the sum of $u_{1}$ can be upper bounded as before. We arrive at
\begin{align*}J_{2}&:=\left| \oint_{\gamma_{2k}^{R}}\phi(w)\int_{\xi_0}^{\infty(\theta)}E_{1}(t\lambda(\xi))
E_{2}(\xi z)\frac{e_{2}(\xi w)}{\xi w}d\xi dw\right|\\
&\le C_{21}\int_{\varepsilon}^{R}\left(|\phi(se^{i\theta_{+}})|+|\phi(se^{i\theta_{-}})|\right)ds \frac{1}{h_{\mathfrak{m}_{1}}\left(\frac{C_{22}}{|t|}\right)}+ C_{23}\int_{\theta_{-}}^{\theta_{+}}|\phi(Re^{i\theta})|d\theta \frac{1}{h_{\mathfrak{m}_{1}}\left(\frac{C_{22}}{|t|}\right)},
\end{align*}
for some positive constants $C_{21},C_{22},C_{23}$.
Since $\phi\in\mathcal{O}^{\mathfrak{m}^{(1/q)}}(\hat{S}((d+\arg(\lambda))/q+2k\pi/\mu))$, it is straightforward to check that the previous expression can be upper bounded by
$$\frac{C_{24}}{h_{\mathfrak{m}^{(1/q)}}\left(\frac{C_{25}}{R}\right)h_{\mathfrak{m}_{1}}\left(\frac{C_{22}}{|t|}\right)},$$
for some $C_{24},C_{25}>0$. Cauchy's theorem allow us to choose $R$ to be $R=|t|^{1/q}$. In addition to this, from property (\ref{equahMpotencia}) one has
$$h_{\mathfrak{m}^{(1/q)}}\left(\frac{C_{25}}{|t|^{1/q}}\right)= \left(h_{\mathfrak{m}}\left(\frac{C_{25}^{q}}{|t|}\right)\right)^{1/q}.$$
If $0<q\le 1$, one can apply property (\ref{e120}) to obtain
$$\left(h_{\mathfrak{m}}\left(\frac{C_{25}^{q}}{|t|}\right)\right)^{1/q}\ge h_{\mathfrak{m}}\left(\frac{C_{25}^{q}}{\rho(1/q)|t|}\right),$$
and if $q\ge1$, $h_{\mathfrak{m}}(s)\le 1$ for all $s\in(0,\infty)$, so that
$$\left(h_{\mathfrak{m}}\left(\frac{C_{25}^{q}}{|t|}\right)\right)^{1/q}\ge h_{\mathfrak{m}}\left(\frac{C_{25}^{q}}{|t|}\right).$$
These facts entail, for some $C_{26}>0$,
$$J_{2}\le \frac{C_{24}}{h_{\mathfrak{m}}\left(\frac{C_{26}}{|t|}\right)h_{\mathfrak{m}_{1}}\left(\frac{C_{22}}{|t|}\right)}.$$
>From the hypotesis (\ref{equammenormuno}) we have, by (\ref{equahdeMequi}), that $h_{\mathfrak{m}}(v)\le Ch_{\mathfrak{m}_1}(Av)$ for every $v>0$, and so, putting $C_{27}=\min\{C_{22}/A,C_{26}\}$, one gets that
$$J_{2}\le \frac{CC_{24}}{h_{\mathfrak{m}}\left(\frac{C_{26}}{|t|}\right)h_{\mathfrak{m}}\left(\frac{C_{22}}{A|t|}\right)}\le \frac{CC_{24}}{\left(h_{\mathfrak{m}}\left(\frac{C_{27}}{|t|}\right)\right)^2}\le
\frac{CC_{24}}{h_{\mathfrak{m}}\left(\frac{C_{27}}{\rho(2)|t|}\right)},$$
where (\ref{e120}) has been used in the last inequality. So, one obtains that $u_1\in\mathcal{O}^{\mathfrak{m}}(\hat{S}(d+2n\pi/\nu)\times D(0,r))$ for some $r>0$ and for $n=0,\ldots,\nu-1$, and the conclusion is immediate.

\end{proof}

Lemmas 6 and 7 in~\cite{Michalik3} can be easily rewritten in our context, and they lead us straightforward to the next result, an analogue of Theorem 3 in~\cite{Michalik3}.

\begin{theo}\label{t278a}
Let $q=\mu/\nu\in\mathbb{Q}$, with $\gcd(\mu,\nu)=1$. Let $m_1(\lambda)$, $m_2(\lambda)$ and $\mathfrak{m}$ be as in Lemma~\ref{lemarelaccrecexpondatossoluc}. If $u(t,z)$ is the solution of (\ref{e116a}), then for every $d\in\R$ the following statements are equivalent:
\begin{enumerate}
\item $\phi\in\mathcal{O}^{\mathfrak{m}^{(1/q)}}(\hat{S}((d+\arg(\lambda))/q+2k\pi/\mu))$ for every $k=0,\ldots,\mu-1$.
\item $u\in\mathcal{O}^{\mathfrak{m}}(\hat{S}(d+2n\pi/\nu)\times D(0,r))$, for $n=0,1,\ldots,\nu-1$.
\end{enumerate}
\end{theo}

%{\color{red}{DESPUES DE DAR LAS EQUIVALENCIAS ENTRE SUMABILIDAD, HAY QUE DAR UNA DEFINICIÓN COMO LA PROPOSICION 2 DE MICHALIK3 (SUMABILIDAD DE FUNCIONES EN UNA VARIABLE, UNIFORME EN LA OTRA VARIABLE, CON NORMA DEL SUP). LOS COMENTARIOS ANTERIORES DIRIGEN A BALSER, PERO ESTA HECHO A LO LARGO DE LAS SECCIONES PREVIAS EN ESTE TRABAJO.}}
%
%\begin{prop}\label{e294a}
%Let $m=(m(p))_{p\ge0}$ be a sequence of moments and $u\in\mathcal{O}(\hat{S})$ for some unbounded sector $S$ in $\C$. If there exist $C_1,C_2>0$ such that
%\begin{equation}\label{e295a}
%|u^{(p)}(0)|\le\frac{C_{1}C_2^{p}p!}{m(p)}, \quad p\in\N_0,
%\end{equation}
%then, $u\in\mathcal{A}^{m}(S)$. Moreover, the bounds in (\ref{equagrowthhMuniform}) remain valid in the whole sector $S$.
%\end{prop}
%\begin{proof}
%Let $u(t)=\sum_{p\ge0}\frac{u_{p}}{p!}t^{p}$ be the Taylor expansion at 0 of the function $u$. Then, for every $t\in S$, one has
%$$|u(t)|=\left|\sum_{p\ge0}\frac{u_{p}}{p!}t^{p}\right|\le \sum_{p\ge0}\frac{|u_{p}|}{p!}t^{p}.$$
%Regarding (\ref{e295a}), one derives
%$$|u(t)|\le \sum_{p\ge0}C_1C_2^{p}\frac{|t|^{p}}{m(p)}\le C_1\sum_{p\ge0}\frac{1}{\inf_{p\in\N_0}m(p)\frac{1}{2C_2|t|^{p}}}\left(\frac{1}{2}\right)^{p},$$
%for every $t\in\hat{S}$. The result immediately follows from the last expression above.
%\end{proof}

Although all the treatment of summability in this paper has been limited to complex valued functions, it can be extended without any difficulty to functions taking their values in a general complex Banach algebra. In particular, we may take this algebra to consist of the bounded holomorphic functions in a fixed neighborhood of the origin in the $z$ plane with the norm of the supremum, and consider summability of formal power series in the $t$ variable with such coefficients. The following definition is natural under this point of view.

\begin{defi}\label{defisumablepornucleo2var}
Let $\hat{u}(t,z)=\sum_{j=0}^\infty u_j(z)t^j$ be a formal series with coefficients in $\mathcal{O}(D(0,r))$ for some $r>0$ (independent of $j$), and let $\mathfrak{m}_e=(m_e(j))_{j\in\N_0}$ be the strongly regular moment sequence of a kernel $e$. We say $\hat{u}$ is $\mathfrak{m}_e-$summable in direction $d\in\R$ if
$$
\hat{T}^{-}_e(t,z)=\sum_{j=0}^\infty \frac{u_j(z)}{m_e(j)}t^j\in\mathcal{O}^{\mathfrak{m}_e}(\hat{S}(d)\times D(0,r)),
$$
where $S(d)$ is an unbounded (small) sector bisected by $d$.
\end{defi}

We are now able to establish a characterization of summability for the formal solutions of (\ref{e116a}) and also for the initial problem (\ref{e7a}),(\ref{e11a}), under appropriate conditions regarding the moment functions involved.

\begin{prop}
Let $\hat{u}$ be a formal solution of (\ref{e116a}). Let $q=\mu/\nu\in\mathbb{Q}$, with $\gcd(\mu,\nu)=1$, and $d\in\R$.  We assume a strongly regular sequence of moments $\mathfrak{m}=(m(p))_{p\in\N_0}$ exists with
\begin{equation}\label{e313a}
m_2(qj)\le C_0A_0^jm(j)m_1(j), \quad j\in\N_0,
\end{equation}
and
\begin{equation}\label{e313abis}
m(j/q)m_1(j/q)\le C_1A_1^jm_2(j), \quad j\in\N_0,
\end{equation}
for suitable $C_0,C_1,A_0,A_1>0$.
%\begin{equation}\label{e313a}
%m(p)m_{1}(p)=m_2(\left\lfloor qp\right\rfloor),
%\end{equation}
%for every $p\in\N_0$.
Then, $\hat{u}$ is $\mathfrak{m}$-summable in direction $d+2n\pi/\nu$ for $n=0,\ldots,\nu-1$ if, and only if, $\phi\in\mathcal{O}^{\mathfrak{m}^{(1/q)}}(\hat{S}((d+\arg\lambda)/q+2k\pi/\mu))$ for $k=0,\ldots,\mu-1$.
\end{prop}

\begin{proof}
Let $n\in\{0,\ldots,\nu-1\}$. By Definition~\ref{defisumablepornucleo2var},
$\hat{u}$ is $\mathfrak{m}$-summable in direction $d+2n\pi/\nu$ if, and only if,
$$v(t,z):=\sum_{j=\beta-1}^{\infty}
\left(\begin{array}{c}j\\\beta-1\end{array}\right)
\frac{\lambda^{j}(\partial_{\mathfrak{m}_2,z})\phi(z)}{m_{1}(j)m(j)}t^{j}
\in\mathcal{O}^{\mathfrak{m}}(\hat{S}(d+2n\pi/\nu)\times D(0,r)).$$
If we put $\tilde{\mathfrak{m}}=(\tilde{m}(p))_{p\ge0}$, with $\tilde{m}(p)=m_1(p)m(p)$, then $\tilde{\mathfrak{m}}$ turns out to be a sequence of moments, as it may be deduced along the same lines as in the Gevrey case (see~\cite[Section 5.8]{balserutx}).
%
%{\color{red}{ESTO ES LO QUE DARA LUGAR A LA MULTISUMABILIDAD. CREO QUE HABRIA QUE DEJARLO INDICADO Y PROFUNDIZAR EN OTRO TRABAJO}}
One can observe that $v$ turns out to be the solution of (\ref{e116a}) when substituting $\mathfrak{m}_1$ by $\tilde{\mathfrak{m}}$. From Theorem~\ref{t278a}, we know that $v\in\mathcal{O}^{\mathfrak{m}}(\hat{S}(d+2n\pi/\nu)\times D(0,r))$ if, and only if, $\phi\in\mathcal{O}^{\mathfrak{m}^{(1/q)}}(\hat{S}((d+\arg\lambda)/q+2k\pi/\mu))$ for $k=0,\ldots,\mu-1$, as desired.
\end{proof}

Finally, we consider the normalised formal solution for (\ref{e7a}) given in~(\ref{equasolunormalizada}). We make the following:
% can be written in the form $\hat{u}=\sum_{\alpha=1}^{\ell}\sum_{\beta=1}^{n_{\alpha}}\hat{u}_{\alpha\beta}$, where $n_{\alpha}\in\N$ for every $\alpha=1,\ldots,\ell$, and $\hat{u}_{\alpha\beta}$ is a formal solution of (\ref{e87a}).

\textbf{Assumption (A):} There exists $q=\mu/\nu\in\mathbb{Q}$ with $gcd(\mu,\nu)=1$ such that $P(\lambda,\xi)$ in (\ref{e30a}) satisfies that
$$\lim_{z\to\infty}\frac{\lambda_{\alpha}(z)}{z^{q}}\in\C\setminus\{0\},$$
for every $\alpha=1,\ldots,\ell$, i.e., $q\in\mathbb{Q}$ is the common order pole of $\lambda_{\alpha}$ for every $\alpha=1,\ldots,\ell$.

The previous results lead to the main result of this last section.
\begin{theo}
Let $d\in\R$. Suppose a strongly regular sequence of moments $\mathfrak{m}=(m(p))_{p\in\N_0}$ exists such that (\ref{e313a}) and (\ref{e313abis}) hold. Let $\hat{u}$ be the normalised formal solution of (\ref{e7a}), (\ref{e11a}). Then, under Assumption (A), $\hat{u}$ is $\mathfrak{m}-$summable in any direction of the form $d+2n\pi/\nu$ for $n=0,\ldots,\nu-1$ if, and only if,
$\phi\in\mathcal{O}^{\mathfrak{m}^{(1/q)}}
(\hat{S}((d+\arg(\lambda_{\alpha\beta}))/q+2k\pi/\mu))$
for every $k=0,\ldots,\mu-1$, every  $\alpha=1,\ldots,\ell$ and every $\beta=1,\ldots,n_{\alpha}$.
\end{theo}

\noindent\textbf{Acknowledgements}: The first and third authors are partially supported by the Spanish Ministry of Science and Innovation under project MTM2009-12561.
The second author is partially supported by the French ANR-10-JCJC 0105 project and the PHC Polonium 2013 project No. 28217SG.
All the authors are partially supported by the Spanish Ministry of Economy and Competitiveness under project MTM2012-31439.

\vskip.5cm
\noindent Authors' Affiliations:\par\vskip.5cm

Alberto Lastra\par
Departamento de F{\'\i}sica y Matem\'aticas\par
Edificio de Ciencias. Campus universitario\par
Universidad de Alcal\'a\par
28871 Alcal\'a de Henares, Madrid, Spain\par
E-mail: alberto.lastra@uah.es\par
\vskip.5cm

St\'ephane Malek\par
UFR de Math\'ematiques Pures et Appliqu\'ees\par
Cit\'e Scientifique - B\^at. M2\par
59655 Villeneuve d'Ascq Cedex, France\par
E-mail: malek@math.univ-lille1.fr\par
\vskip.5cm

Javier Sanz\par
Departamento de \'Algebra, An\'alisis Matem\'atico, Geometr{\'\i}a y Topolog{\'\i}a\par
Instituto de Investigaci\'on en Matem\'aticas de la Universidad de Valladolid, IMUVA\par
Facultad de Ciencias\par
Universidad de Valladolid\par
47011 Valladolid, Spain\par
E-mail: jsanzg@am.uva.es

\end{document}